\documentclass{amsart}
\usepackage[all]{xy}
\usepackage{amsfonts,amssymb,amsmath,amscd}
\usepackage{amsthm}

\newcounter{zlist}
\newenvironment{zlist}{\begin{list}{{\rm(\arabic{zlist})}}{
\usecounter{zlist}\leftmargin2.5em\labelwidth2em\labelsep0.5em
\topsep0.6ex\itemsep0.3ex plus0.2ex minus0.3ex
\parsep0.3ex plus0.2ex minus0.1ex}}{\end{list}}

\newcounter{blist}
\newenvironment{blist}{\begin{list}{{\rm(\alph{blist})}}{
\usecounter{blist}\leftmargin2.5em\labelwidth2em\labelsep0.5em
\topsep0.6ex \itemsep0.3ex plus0.2ex minus0.3ex
\parsep0.3ex plus0.2ex minus0.1ex}}{\end{list}}

\newcounter{rlist}
\newenvironment{rlist}{\begin{list}{{\rm(\roman{rlist})}}{
\usecounter{rlist}\leftmargin2.5em\labelwidth2em\labelsep0.5em
\topsep0.6ex\itemsep0.3ex plus0.2ex minus0.3ex
\parsep0.3ex plus0.2ex minus0.1ex}}{\end{list}}

\swapnumbers
\newtheorem{theorem}{Theorem}[section]
\newtheorem{lemma}[theorem]{Lemma}
\newtheorem{thm}[theorem]{}
\newtheorem{proposition}[theorem]{Proposition}

\newtheorem{corollary}[theorem]{Corollary}
\newtheorem{remark}[theorem]{Remark}
\numberwithin{equation}{section}

\newcommand{\A}{{\mathbb {A}}}
\newcommand{\B}{{\mathbb {B}}}
\newcommand{\C}{{\mathbb {C}}}
\newcommand{\G}{\mathbf{G}}
\newcommand{\D}{{\mathbb {D}}}

\newcommand{\V}{{\mathbb {V}}}
\newcommand{\X}{\mathbb{X}}
\newcommand{\Y}{\mathbb{Y}}
\newcommand{\bV}{{\mathbb {V}}}

\newcommand{\cC}{{\mathcal {C}}}

 \newcommand{\bF}{\mathbf{F}}
\newcommand{\bG}{\mathbf{G}}

\newcommand{\oH}{{\overline{H}}}
\newcommand{\woH}{{\widehat{\overline{H}}}}
\newcommand{\oR}{{\overline{R}}}
\newcommand{\oF}{{\overline{F}}}

\newcommand{\wG}{{\widehat{G}}}

\newcommand{\uH}{{\underline{H}}}

\newcommand{\bH}{{\bf{H}}}
\newcommand{\bT}{{\bf{T}}}
\newcommand{\ve}{\varepsilon}
\newcommand{\vth}{\vartheta}
\newcommand{\vareps}{\varepsilon}
\newcommand{\lra}{\longrightarrow}
\newcommand{\ot}{\otimes}

\newcommand{\II}{\mathbb{I}}

\newcommand{\Comon}{{\rm Comon}}
\newcommand{\End}{{\rm End}}
\newcommand{\Hom}{{\rm Hom}}

\newcommand{\Mon}{{\rm Mon}}
\newcommand{\Mor}{{\rm Mor}}
\newcommand{\Bimon}{{\rm Bimon}}
\newcommand{\Set}{{\rm Set}}
\newcommand{\Inj}{{\rm {\bf Inj}}}
\newcommand{\Proj}{{\rm Proj}}

\newcommand{\nG}{G}

\begin{document}

\title{Galois functors and entwining structures}
 \author{Bachuki Mesablishvili, Robert Wisbauer}

\begin{abstract}
{\em Galois comodules} over a coring can be characterised by properties of the
relative injective comodules.
They motivated the definition of
{\em Galois functors} over some comonad (or monad) on any category and in the
first section of the present paper we investigate the role of the relative injectives (projectives) in this context.

Then we generalise the notion of corings (derived from an
entwining of an algebra and a coalgebra) to the entwining of a
monad and a comonad. Hereby a key role is played by the notion of
a {\em grouplike natural transformation} $g:I\to G$ generalising
the grouplike elements in corings. We apply the evolving theory to
Hopf monads on arbitrary categories, and to comonoidal functors on
monoidal categories in the sense of A. Brugui\`{e}res and A.
Virelizier. As well-know, for any set $G$ the product $G\times-$
defines an endofunctor on the category of sets and this is a Hopf
monad if and only if $G$ allows for a group structure. In the
final section the elements of this case are generalised to
arbitrary categories with finite products  leading to {\em Galois
objects} in the sense of Chase and Sweedler.
\smallskip

Key Words: Corings, (Galois) comodules, Galois functors,
relative injectives (projectives), equivalence of categories.

AMS classification: 18A40,  16T15.
\end{abstract}

\maketitle

\tableofcontents

\section*{Introduction}

For any ring $A$, consider an $A$-coring $\cC$ and a right
$\cC$-comodule $P$ with $S=\End^\cC(P)$. Then, in \cite{WisGal},
$P$ is called a {\em Galois comodule} provided the natural
transformation $\Hom_A(P,-)\ot_S P \to -\ot_A \cC$ is an
isomorphism. Such modules can be characterized by properties of
the $(\cC,A)$-injective comodules \cite[4.1]{WisGal}.

This notion was extended to {\em comodule functors} for some
comonad in \cite[3.5]{MW} and, in the first section of the present
paper, properties of relative injective objects will be
investigated. Dually, module functors are considered for some {\em
monad functors} leading to the study of relative projectives.

In Section 2 the interplay of the Galois property for entwining
structures (mixed distributive laws) is studied, while in Section
3 the characterizations of Galois corings in module categories
(e.g. \cite[28.18]{BW}) is transferred to entwining structures on
any categories. Applying this to bimonads in the sense of
\cite{MW} leads to new characterizations of Hopf monads (Section
4). As another application we consider bimonads in the sense of
Brugui\`eres and Virelizier \cite{BV} and eloborate the relation
between the different approaches (Section 5).

In the final section we generalise known properties of the
endofunctors $G\times-$ on the category of sets, $G$ any set, to
categories with finite products. This relates our notions with
{\em Galois objects} in the sense of Chase and Sweedler \cite{CS}
(in the category opposite to commutative algebras) and we
obtain a more general form of their Theorem 12.5 by
replacing the condition on the Hopf algebra to be
finitely generated and projective over the base ring by flatness without
finiteness condition.

 \section{Galois comodule and module functors}

Let $\A$ and $\B$ denote any categories.
Recall (e.g. from \cite{EM}) that a {\em monad} $\bT$ on $\A$ is a
triple $(T,m,e)$ where $T:\A\to \A$ is a functor with natural
transformations $m:TT\to T$, $e:I\to T$ satisfying associativity
and unitality conditions. A {\em $T$-module} is an object $a\in
\A$ with a morphism $h_a:T(a)\to a$ subject to associativity and
unitality conditions. The (Eilenberg-Moore) category of
$\bT$-modules is denoted by $\A_T$ and there is a free functor
$\phi_T:\A\to \A_T,\; a\mapsto (T(a),m_a)$ which is left adjoint
to the forgetful functor $U_T:\A_T\to \A$.

Dually, a {\em comonad} $\bG$ on $\A$ is a triple $(G,\delta,\ve)$
where $G:\A\to \A$ is a functor with natural transformations
$\delta:G\to GG$, $\ve:G\to I$, and {\em $G$-comodules} are
objects $a\in \A$ with morphisms $\rho_a:a\to G(a)$. Both notions
are subject to coassociativity and counitality conditions. The
(Eilenberg-Moore) category of $\bG$-comodules is denoted by $\A^G$
and there is a cofree functor $\phi^G:\A\to \A^G,\; a\mapsto
(G(a),\delta_a)$ which is right adjoint to the forgetful functor
$U^G:\A^G\to \A$.

For convenience we recall some notions from \cite[Section 3]{MW}.

\begin{thm}\label{com-fun}{\bf $\bG$-comodule functors.} \em
Given a comonad $\textbf{G}=(G,\delta,\varepsilon)$ on
$\A$, a functor $F: \B \to \A$ is a {\em left $\textbf{G}$-comodule}
if there exists a natural transformation $\beta: F \to GF$ with commutative
 diagrams

\begin{equation} \xymatrix{
F \ar@{=}[dr] \ar[r]^-{\beta}& GF \ar[d]^-{\varepsilon F}\\
& F,} \qquad \xymatrix{
F \ar[r]^-{\beta} \ar[d]_-{\beta}& GF \ar[d]^-{\delta F}\\
GF \ar[r]_-{G \beta}& GGF.}
\end{equation}
Obviously $(G,\delta)$ and $(GG, \delta G)$
both are left $\textbf{G}$-comodules.

A $\bG$-comodule structure on $F : \B \to \A$ is equivalent to
the existence of a functor (dual to \cite[Proposition II.1.1]{D})
$\overline{F} : \B \to \A^G$ leading to a commutative diagram
$$\xymatrix{ \B \ar[r]^\oF \ar[dr]_F & \A^G \ar[d]^{U^G}\\
  & \A .}$$

Indeed,
if $\overline{F}$ is such
a functor, then $\overline{F}(b)=(F(b), \beta_{b})$ for some morphism
$\beta_{b}: F(b) \to GF(b)$
and the collection $\{\beta_b,\, b \in\B\}$
constitutes a natural transformation $\beta:F \to GF$  making $F$ a $\textbf{G}$-comodule.
Conversely, if $(F,\beta: F\to GF)$
is a $\textbf{G}$-module, then $\overline{F} :
\B \to \A^G$ is defined by $\overline{F}(b)=(F(b), \beta_b)$.

 If a $\textbf{G}$-comodule $(F,\beta)$
 admits a right adjoint $R: \A \to \B$, with
counit $\sigma : FR \to 1$, then the composite
$$\xymatrix{t_{\overline{F}}:FR \ar[r]^-{\beta R}& GFR
\ar[r]^-{G \sigma }  & G}$$
is a comonad morphism from the comonad
generated by the adjunction $F \dashv R$ to the comonad $\textbf{G}$.
\end{thm}

\begin{proposition}\label{P.1.7} {\rm (\cite[Theorem 4.4]{M})} The functor $\overline{F}$ is an equivalence of categories
if and only if the functor $F$ is comonadic and $t_{\overline{F}}$ is an
isomorphism of comonads.
\end{proposition}

\begin{thm}{\bf Definition.}\label{def-galois} \em (\cite[Definition 3.5]{MW})
 A left $\textbf{G}$-comodule $F: \B \to \A$ with a right
adjoint $R: \A \to \B$ is said to be $\textbf{G}$-\emph{Galois} if the
corresponding morphism $t_{\overline{F}}: FR \to G$ of comonads on
$\A$ is an isomorphism.
\end{thm}

Thus, $\oF$ is an equivalence if and only if $F$ is $\bG$-Galois
and comonadic.

\begin{thm}\label{right-adj}{\bf Right adjoint for $\oF$.} \em
When the category $\B$ has equalisers of coreflexive pairs, the
functor $\overline{F}$ has a right adjoint which can be described
as follows (see \cite{D}): With the composite
$$\gamma:\xymatrix{R
\ar[r]^-{\eta R}& RFR \ar[r]^{R t_{\overline{F}}}& RG,}$$ a right
adjoint to $\overline{F}$ is the equaliser $(\overline{R},
\overline{e})$ of the diagram
$$\xymatrix{ R U^G
\ar@{->}@<0.5ex>[rr]^-{RU^G \eta^G} \ar@ {->}@<-0.5ex> [rr]_-{\gamma
U^G}&& RGU^G=RU^G \phi^GU^G,}$$
with $\eta^G:1\to\phi^GU^G$ the unit of $U^G \dashv \phi^G$.

An easy inspection shows that for any $(a, \theta_a)\in \A^G$, the
$(a, \theta_a)$-component of the above diagram is
$$\xymatrix{ R(a) \ar@{->}@<0.5ex>[r]^-{R(\theta_a)} \ar@
{->}@<-0.5ex> [r]_-{\gamma_a}& RG(a).}$$
\end{thm}

 Now, for any $a\in \A$,
$(\overline{R}(\overline{F}))(a)$ can be seen as the equaliser

$$\xymatrix{(\overline{R}(\overline{F}))(a)\ar[r]^-{\overline{e}_{\overline{F}(a)}} &
RF(a) \ar@{->}@<0.5ex>[r]^-{R(\beta_a)} \ar@ {->}@<-0.5ex>
[r]_-{\gamma_{F(a)}}& RGF(a).}$$ Thus, writing $P$ for the monad on
$\A$ generated by the adjunction $\overline{F} \dashv \overline{R}$,
the diagram
$$\xymatrix{ P \ar[r]^-{\overline{e}}& RF\ar@{->}@<0.5ex>[r]^-{R\beta} \ar@
{->}@<-0.5ex> [r]_-{\gamma F}& RGF}$$
is an equalier diagram.
\smallskip

In view of the characterization of Galois functors we have a
closer look at some related classes of relative injective objects.

Let $F : \B \to \A$ be any functor. Recall (from \cite{Sob}) that an object
$b \in \B$ is said to be $F$-\textit{injective} if for any diagram
in $\B$,
$$\xymatrix{b_1 \ar[d]_g \ar[r]^f & b_2 \ar@{.>}[dl]^h\\
b&}$$
with $F(f)$ a split monomorphism in $\A$, there exists a
morphism $h: b_2 \to b$ such that $hf=g$. We write
$\Inj (F, \B)$ for the full subcategory of $\B$ with
objects all $F$-injectives.

The following result from \cite{Sob} will be needed.

\begin{proposition}\label{P.3.3}
Let $\eta, \varepsilon: F \dashv R : \A \to \B$ be an adjunction.
For any object $b \in \B$, the following assertions are
equivalent:

\begin{blist}
\item  $b$ is $F$-injective;
\item $b$ is a coretract for some $R(a)$, with $a \in \A$;
\item the $b$-component $\eta_b : b \to RF(b)$ of $\eta$ is
      a split monomorphism.
\end{blist}
\end{proposition}

\begin{remark}\label{R.3.4} \em
For any $a \in ¸\A$, $R(\varepsilon_a)\cdot \eta_{R(a)}=1$ by one
of the triangular identities for the adjunction $F \dashv R$.
Thus, $R(a) \in \Inj(F, \B)$ for all $a \in \A$.
Moreover, since the composite of coretracts is again a coretract,
it follows from (b) that $\Inj(F, \B)$ is closed
under coretracts.
\end{remark}

\begin{thm}\label{P.3.5}{\bf Functor between injectives.} \em
Let $F:\B\to \A$ be a $\bG$-module with a right adjoint $R:\A\to \B$.
and unit $\eta:I\to RF$.
Write $\bG'$ for the comonad on $\A$ generated by the adjunction
$F\dashv R$ and consider the comparison functor
  $K_{G'}: \B \to \A^{\nG'}$.
 If $b\in \B$ is $F$-injective, then $K_{G'}(b)=(F(b), F(\eta_b))$ is
$U_{G'}$-injective, since by the fact
that $\eta_b$ is a split monomorphism in $\B$,
 $(\eta_{G'})_{\phi^{G'}(b)}=F(\eta_b)$
is a split monomorphism in $\A^{\nG'}$.
Thus the functor $K_{G'}: \B \to \A_{G'}$ yields a functor
$$ \Inj (K_{G'}):  \Inj(\bF,\B) \to \Inj(\phi^{G'},\A^{\nG'}).$$
  When $\B$ has equalisers, this functor 
is an equivalence of categories (see \cite{Sob}).
\end{thm}

We shall henceforth assume that $\B$ has equalisers.

\begin{proposition}\label{P.3.6} The functor $\overline{R}: \A^\nG \to \B$
restricts to a functor
$$\overline{R}': \Inj(U^{G},\A^{\nG}) \to \Inj(F, \B).$$
\end{proposition}
\begin{proof} Let $(a, \theta_a)$ be an arbitrary object of
$\Inj (U^{G}, \A^{\nG})$. Then, by Proposition \ref{P.3.3},
there exists an object $a_0 \in \A$ such that $(a, \theta_a)$ is a
coretraction of $\phi^G(a_0)=(G(a_0), \delta_{a_0})$ in $\A^\nG$, i.e.,
there exist morphisms
\begin{center}
$f: (a, \theta_a) \to (G(a_0),
\delta_{a_0})$ and $g: (G(a_0), \delta_{a_0}) \to (a, \theta_a)$
\end{center}
in $\A^\nG$ with $gf=1$. Since $f$ and $g$ are morphisms in $\A^\nG$,
the diagram
$$\xymatrix{
G(a_0) \ar@{->}@<+1.5ex>[d]^-{g} \ar[r]^{(\delta_G)_{a_0}} &
GG(a_0)\ar@{->}@<+1.5ex>[d]^-{G(g)}\\
a \ar@{->}@<-0.5ex>[u]^-{f} \ar[r]_{\theta_a}&
G(a)\ar@{->}@<-0.5ex>[u]^-{G(f)} }
$$commutes. By naturality of $\gamma$ (see \ref{right-adj}), the  diagram
$$\xymatrix{
RG(a_0) \ar@{->}@<+1.5ex>[d]^-{R(g)} \ar[rr]^{\gamma_{G(a_0)}}
&&RGG(a_0)
\ar@{->}@<+1.5ex>[d]^-{RG(g)}\\
R(a) \ar@{->}@<-0.5ex>[u]^-{R(f)} \ar[rr]_{\gamma_a}&&
RG(a)\ar@{->}@<-0.5ex>[u]^-{RG(f)} }$$
also commutes. Consider now the following commutative diagram
\begin{equation}\label{D.3.3}\xymatrix{R(a_0) \ar@{.>}@<+1.5ex>[dd] \ar[r]^{\gamma_{a_0}}&
RG(a_0) \ar@{->}@<+1.5ex>[dd]^-{R(g)}
\ar@{->}@<0.5ex>[rr]^{\gamma_{G(a_0)}}
\ar@{->}@<-0.5ex>[rr]_{R((\delta_G)_{a_0})} &&RGG(a_0)\ar@{->}@<+1.5ex>[dd]^-{RG(g)}\\\\
\overline{R}(a, \theta_a)\ar@{.>}@<-0.5ex>[uu]
\ar[r]_{\overline{e}_{(a, \theta_a)}}&R(a)
\ar@{->}@<-0.5ex>[uu]^-{R(f)} \ar@{->}@<0.5ex>[rr]^{\gamma_a}
\ar@{->}@<-0.5ex>[rr]_{R(\theta_a)}&&
RG(a)\ar@{->}@<-0.5ex>[uu]^-{RG(f)} .} \end{equation}
 It is not
hard to see that the top row of this diagram is a (split)
equaliser (see \cite{G}), and since the bottom row is an
equaliser by the very definition of $\overline{e}$, it follows
from the commutativity of the diagram that $\overline{R}(a,
\theta_a)$ is a coretract of $R(a_0)$, and thus is an object of
$\Inj (F, \B)$ (see Remark \ref{R.3.4}). It means that the
functor $\overline{R}: \A^\nG \to \B$ can be restricted to a functor
$\overline{R}': \Inj (U^G, \A^\nG)\to
\Inj (F, \B)$.
\end{proof}

\begin{proposition}\label{P.3.7}
Suppose that for any $b \in \B$, $(t_{\overline{F}})_{F(b)}$ is an
isomorphism. Then the functor $\overline{F}: \B \to \A^\nG$ can be
restricted to a functor
$$\overline{F}': \Inj (F, \B)\to \Inj (U^G, \A^\nG).$$
\end{proposition}

\begin{proof} Let $\delta'$ denote the comultiplication in the
comonad $\G'$ (see \ref{P.3.5}).
 Then for any $b \in \B$,
$$\begin{array}{rl}
\overline{F}(RF(b))&=\A_{t_{\overline{F}}}(\phi^{G'}(UF(b)))
   =\A_{t_{\overline{F}}}(FRF(b), F \eta_{RF(b)})\\[+1mm]
& =\A_{t_{\overline{F}}}(G'F(b), \delta'_{F(b)})=
 (G'F(b),(t_{\overline{F}})_{G'F(b)}\cdot\delta'_{F(b)}).
\end{array}$$
Consider now the diagram
$$
\xymatrix{ G'F(b) \ar[rrr]^-{(t_{\overline{F}})_{F(b)}}
\ar[d]_{\delta'_{F(b)}}&&& GF(b) \ar[ddd]^{\delta_{F(b)}}\\
G'G'F(b)
\ar@{}[rrru]_{(1)}\ar[rrrdd]^{(t_{\overline{F}})_{F(b)}.(t_{\overline{F}})_{F(b)}}
\ar[dd]_{(t_{\overline{F}})_{G'F(b)}} &&& &\\\\
GG'F(b) \ar[rrr]_{G((t_{\overline{F}})_{F(b)})}&&& GGF(b)\, ,}
$$ in which the triangle commutes by the 
definition of the
composite $(t_{\overline{F}})_{F(b)}.(t_{\overline{F}})_{F(b)}$, while the
diagram (1) commutes since $t_{\overline{F}}$ is a morphism of
comonads.
The commutativity of the outer diagram shows that
$(t_{\overline{F}})_{F(b)}$ is a morphism from the $G$-coalgebra
$\overline{F}(RF(b))=(G'F(b),(t_{\overline{F}})_{G'F(b)}\cdot
\delta'_{F(b)})$ to the $G$-coalgebra $(GF(b), \delta_{F(b)})$.
Moreover, $(t_{\overline{F}})_{F(b)}$ is an isomorphism by our
assumption. Thus, for any $b \in \B$, $\overline{F}(RF(b))$ is
isomorphic to the $G$-coalgebra $(GF(b), \delta_{F(b)})$, which is
of course an object of the category $\Inj (U^G,
\A^\nG)$. Now, since any $b \in \Inj (F, \B)$ is a
coretract of $RF(b)$ (see Remark \ref{R.3.4}), and since any functor takes
coretracts to coretracts, it follows that, for any $b \in
\Inj (F, \B)$, $\overline{F}(b)$ is a coretract of
the $G$-coalgebra $(GF(b), \delta_{F(b)})\in
\Inj (U^G, \A^\nG)$, and thus is an object of the
category $\Inj (U^G, \A^\nG)$ again by Remark \ref{R.3.4}.
This completes the proof.
\end{proof}

The following technical observation is needed for the next proposition.

\begin{lemma}\label{L.3.8} Let $\iota, \kappa : W \dashv W' : \Y \to \X$ be an
adjunction of any categories. If $i: x' \to x$ and $j: x \to x'$ are morphisms in
$\X$ such that $ji=1$ and if $\iota_x$ is an isomorphism, then
$\iota_{x'}$ is also an isomorphism.
\end{lemma}

\begin{proof} Since $ji=1$, the  diagram
$$\xymatrix{ x'\ar[r]^{i}&x
\ar@{->}@<0.5ex>[r]^-{1} \ar@ {->}@<-0.5ex> [r]_-{ij}& x}$$ is a
split equaliser. Then the diagram $$\xymatrix{ W'W(x')
\ar[rr]^-{W'W(i)}&& W'W(x) \ar@{->}@<0.5ex>[rr]^-{1} \ar@
{->}@<-0.5ex> [rr]_-{W'W(ij)}&& W'W(x)}$$ is also a split
equaliser. Now considering the following commutative diagram

$$
\xymatrix{x' \ar@{.>}[d]_{\iota_{x'}}\ar[rr]^{i}&& x
\ar[d]_{\kappa_x} \ar@{->}@<0.5ex>[rr]^-{1} \ar@
{->}@<-0.5ex> [rr]_-{ij}&& x \ar[d]^{\kappa_x}\\
W'W(x') \ar[rr]_-{W'W(i)}&& W'W(x) \ar@{->}@<0.5ex>[rr]^-{1} \ar@
{->}@<-0.5ex> [rr]_-{W'W(ij)}&& W'W(x)}$$ and recalling that the
vertical two morphisms are both isomorphisms by assumption, we get
that the morphism $\iota_{x'}$ is also an isomorphism.
\end{proof}

\begin{proposition}\label{P.3.9} In the situation of Proposition \ref{P.3.7},
$\Inj (F,\B)$ is (isomorphic to) a coreflective subcategory of the category
$\Inj (U^G, \A^\nG)$.
\end{proposition}

\begin{proof} By Proposition \ref{P.3.6}, the functor
$\overline{R}$ restricts to a functor
$$\overline{R}':\Inj (U^G, \A^\nG) \to \Inj (F, \B),$$
while according to Proposition \ref{P.3.7},
the functor $\overline{F}$ restricts to a functor
$$\overline{F}':\Inj (F, \B) \to \Inj (U^G,\A^\nG).$$
Since
\begin{itemize}
\item $\overline{F}$ is a left adjoint to $\overline{R}$,
\item $\Inj (F, \B)$ is a full subcategory of $\B$, and
\item $\Inj (U^G, \A^\nG)$ is a full subcategory of $\A^\nG$,
\end{itemize}
the functor $\overline{F}'$ is left adjoint
to the functor $\overline{R}'$, and the unit $\overline{\eta}': 1
\to \overline{R}' \overline{F}'$ of the adjunction $\overline{F}'
\dashv \overline{R}'$ is the restriction of $\overline{\eta}:
\overline{F} \dashv \overline{R}$ to the subcategory
$\Inj (F, \B)$, while the counit
$\overline{\varepsilon}': \overline{F}' \overline{R}' \to 1$ of
this adjunction is the restriction of  $\overline{\varepsilon}:
\overline{F}\overline{R}\to 1$ to the subcategory
$\Inj (U^G, \A^\nG)$.

Next, since the top of the diagram \ref{D.3.3} is a (split)
equaliser, $\overline{R}(G(a_0), \delta_{a_0})\simeq R(a_0)$. In
particular, taking $(GF(b), \delta_{F(b)})$, we see that
$$RF(b) \simeq
\overline{R}(GF(b), \delta_{F(b)})= \overline{R} \,
\overline{F}(UF(b)).$$
 Thus, the $RF(b)$-component
$\overline{\eta}'_{RF(b)}$ of the unit $\overline{\eta}': 1 \to
\overline{R}' \overline{F}'$ of the adjunction $\overline{F}'
\dashv \overline{R}'$ is an isomorphism. It now follows from Lemma
\ref{L.3.8} - since any $b \in \Inj (F, \B)$ is a
coretraction of $RF(b)$ - that $\overline{\eta}'_b$ is an
isomorphism for all $b \in \Inj (F, \B)$ proving
that the unit $\overline{\eta}'$ of the adjunction $\overline{F}'
\dashv \overline{R}'$ is an isomorphism. Thus
$\Inj (F, \B)$ is (isomorphic to) a coreflective
subcategory of the category $\Inj (U^G, \A^\nG)$.
\end{proof}

\begin{corollary}\label{C.3.10} In the situation of Proposition \ref{P.3.7},
suppose that each component of the unit $\eta: 1 \to RF$ is a split
monomorphism. Then the category $\B$ is (isomorphic to) a
coreflective subcategory of $\Inj (U^G, \A^\nG)$.
\end{corollary}

\begin{proof} When each component of the unit $\eta: 1 \to RF$ is a split
monomorphism, it follows from Proposition \ref{P.3.3} that every $b \in
\B$ is $F$-injective; i.e. $\B=\Inj (F, \B)$. The
assertion now follows from Proposition \ref{P.3.9}.
\end{proof}

\begin{thm}\label{T.3.11}{\bf Characterization of $\G$-Galois comodules.}
Assume $\B$ to admit equalisers,  let $\G$ be a comonad on $\A$,
and $F:\B\to \A$ a functor with right adjoint $R:\A\to \B$. If
there exists a functor $\overline{F} : \B \to \A^{\nG}$ with
$U^{\nG}\overline{F}=F$, then the following are equivalent:
\begin{blist}
\item  $F$ is $\G$-Galois, i.e. $t_{\overline{F}}: \textbf{G}\,'\to
        \textbf{G}$ is an isomorphism;
\item the following
composite is an isomorphism:
$$\xymatrix{ \overline{F} R \ar[r]^-{\eta_G
\overline{F}R} &\phi^G U^G \overline{F}R=\phi^G FR \ar[r]^-{\phi^G
\varepsilon}& \phi^G};$$
\item the functor
$\overline{F}: \B \to \A^{\nG}$ restricts to an equivalence of
categories $$\Inj (F, \B) \to
\Inj (U^{G}, \A^{\nG});$$
\item  for any $(a,\theta_a) \in \Inj (U^{G}, \A^{\nG})$, the $(a,
\theta_a)$-component $\overline{\varepsilon}_{(a, \theta_a)}$ of
the counit $\overline{\varepsilon}$ of the adjunction
$\overline{F} \dashv \overline{R}$, is an isomorphism;
\item
for any $a \in \A$,
$\overline{\varepsilon}_{\phi_{G}(a)}=\overline{\varepsilon}_{(G(a),
\delta_a)}$ is an isomorphism.
\end{blist}
\end{thm}

\begin{proof}
That (a) and (b) are equivalent is proved in \cite{D}.
By the proof of \cite[Theorem of 2.6]{G},
for any $a \in \A$,
$\overline{\varepsilon}_{\phi^G (a)}=\overline{\varepsilon}_{(G(a),
\delta_a)}=(t_{\overline{F}})_a$, thus (a) and (e) are equivalent.

By Remark \ref{R.3.4}, (d) implies (e).

Since $\B$ admits equalisers by our assumption on $\B$, it follows
from Proposition \ref{P.3.5} that the functor $\Inj (K_{G'})$ is
an equivalence of categories. Now, if $t_{\overline{F}}:
\textbf{G} ' \to \textbf{G}$ is an isomorphism of comonads, then
the functor $\A_{t_{\overline{F}}}$ is an isomorphism of
categories, and thus $\overline{F}$ is isomorphic to the
comparison functor $K_{G'}$. It now follows from Proposition
\ref{P.3.5} that $\overline{F}$ restricts to the functor $\Inj (F,
\B) \to \Inj (U^{G},\A^{\nG})$
 which is an equivalence of categories.
Thus (a) $\Rightarrow$ (c).

If the functor $\overline{F}: \B \to \A^\nG$ restricts to a functor
$$\overline{F}':\Inj (F, \B) \to \Inj (U^{G},\A^{\nG}),$$
then one can prove as in the proof of Proposition \ref{P.3.9}
that $\overline{F}'$ is left adjoint to $\overline{R}'$ and that
the counit $\overline{\varepsilon}': \overline{F}' \,
\overline{R}' \to 1$ of this adjunction is the restriction of the
counit $\overline{\varepsilon}: \overline{F}\, \overline{R} \to 1$
of the adjunction $\overline{F} \dashv \overline{R}$ to the
subcategory $\Inj (U^{G}, \A^{\nG})$. Now, if
$\overline{F}'$ is an equivalence of categories, then
$\overline{\varepsilon}'$ is an isomorphism. Thus, for any $(a,
\theta_a) \in \Inj (U^{G}, \A^{G})$,
$\overline{\varepsilon}'_{(a, \theta_a)}$ is an isomorphism
proving that (c)$\Rightarrow$(d).
\end{proof}

\begin{thm}\label{T-mod}{\bf  $\bT$-module functors.} \em
Given a monad $\textbf{T}=(T,m,e)$ on $\A$, a functor $R : \B
\to \A$ is said to be a {\em (left) $\textbf{T}$-module} if
there exists a natural transformation $\alpha : TR \to R$ with
commuting diagrams
\begin{equation}
\xymatrix{
R \ar@{=}[dr] \ar[r]^-{eR}&TR \ar[d]^-{\alpha}\\
& R ,} \qquad \xymatrix{
TTR \ar[r]^-{mR } \ar[d]_-{T \alpha}& TR \ar[d]^-{\alpha}\\
TR \ar[r]_-{\alpha}& R.}
\end{equation}
 It is easy to see that $(T,m)$ and $(TT, mT)$
both are left $\textbf{T}$-modules.

A $\bT$-module structure on $R$ is equivalent to the existence of a functor
$\overline{R} : \B \to \A_T$ inducing a commutative diagram
(see \cite[Proposition II.1.1]{D})
$$\xymatrix{ \B \ar[r]^\oR \ar[dr]_R & \A_T \ar[d]^{U_T} \\
           & \A.}$$


Indeed (compare \cite {D}), if $\overline{R}$ is such
a functor, then $\overline{R}(b)=(R(b), \alpha_{b})$ for some morphism
$\alpha_{b}: TR(b) \to R(b)$ and the collection $\{\alpha_b,\, b \in
\B\}$ constitutes a natural transformation $\alpha:TR \to R$  making $R$ a $\textbf{T}$-module.
Conversely, if $(R,
\alpha: TR \to R)$ is a $\textbf{T}$-module, then $\overline{R} : \B
\to \A_T$ is defined by $\overline{R}(b)=(R(b), \alpha_b)$.

 For any $\textbf{T}$-module
$(R: \B \to\A,\alpha)$ admitting a left adjoint functor $F : \A \to \B$, the
composite
$$ t_{\oR}: \xymatrix{T \ar[r]^-{T \eta}& TRF \ar[r]^-{\alpha F} & RF},$$
where $\eta : 1 \to RF$ is the unit of the adjunction $F \dashv R$,
is a monad morphism from $\textbf{T}$ to the monad on $\A$
generated by the adjunction $F \dashv R$.
Dual to \cite[Lemma 4.3]{M}, we have a commutative diagram
 $$\xymatrix{ \B \ar[r]^{K_R} \ar[dr]_\oR & \A_{RF} \ar[d]^{\A_{t_\oR}} \\
           & \A_T,}$$
with the comparison functor $K_R:\B \to \A_{RF},\; b\mapsto (R(b),R(\ve_b))$,
where $\ve$ is the counit of the adjunction $F \dashv R$.
 As the dual of \cite[Theorem 4.4]{M}, we have
\end{thm}
\begin{proposition}\label{P.1.2}
The functor $\overline{R}$ is an equivalence of categories if and
only if the functor $R$ is monadic (i.e. $K_R$ is an
equivalence) and $t_{\overline{R}}$ is an isomorphism of
monads.
\end{proposition}

Similar to \ref{com-fun} one defines (\cite[Definition 3.5]{MW}, \cite[2.19]{BBW})

\begin{thm}\label{D.3.5}{\bf Definition.} \em
A left $\textbf{T}$-module $R: \B \to \A$ with a left
adjoint $F: \A \to \B$ is said to be $\textbf{T}$-\emph{Galois} if the
corresponding morphism $t_{\overline{R}}: T \to RF$ of monads on
$\A$ is an isomorphism.
\end{thm}

Given a functor $R: \B \to \A$, we write $\Proj (R, \B)$ for
the full subcategory of $\B$ given by $R$-projective objects. The
following is dual to \ref{T.3.11}. 

\begin{theorem}\label{Char-T-Gal} {\bf Characterization of $\bT$-Galois modules.}
Assume the category $\B$ to have equalisers. Let $\textbf{T}=(T, m,
e)$ be a monad on $\A$, and   $R : \B \to \A$ a left $\bT$-module functor with left
adjoint $F: \A \to \B$ (and unit $\eta$, counit $\ve$).
If there exists a functor $\overline{R} : \B
\to \A_T$ with $U_T \overline{R}=R$, then the following are
equivalent:
\begin{blist}
    \item   $R$ is $\textbf{T}$-Galois;
    \item  the following composition is an isomorphism:
    $$\xymatrix{\phi_T \ar[r]^-{\phi_T \eta}& \phi_T R F=\phi_T U_T \overline{R}F
    \ar[r]^-{\varepsilon_T \overline{R}F}& \overline{R}F};$$
    \item  the functor $\overline{R}: \B \to \A_T$
    restricts to an equivalence between the categories $ \Proj (R, \B)$
      and $\Proj (U_T, \A_T)$;
    \item  for any $(a, h_a)\in \Proj (U_T, \A_T)$,
     the $(a, h_a)$-component of the unit $\overline{\eta}$ of the
    adjunction $\overline{L} \dashv \overline{R}$, is an isomorphism;
    \item for any $a \in \A$, $\overline{\eta}_{\phi_T(a)}=\overline{\eta}_{(T(a),
    \,m_a)}$ is an isomorphism.
\end{blist}
\end{theorem}

\begin{thm}{\bf Left adjoint for $\oR$.} \em
Let $(R, \alpha: TR \to R)$ be a left $\textbf{T}$-module with a
left adjoint $F:\B \to \A$. Consider the composite
$$\beta :\xymatrix{ FT \ar[r]^-{Ft_{\overline{R}}} & FRF
\ar[r]^-{\varepsilon F}&F},$$
where $\varepsilon :FR \to 1$ is the
counit of $F \dashv R$. It is easy to check that $(F,
\beta)$ is a right $\textbf{T}$-module. According to \cite[Theorem
A.1]{D}, when a coequaliser $(R, i)$ exists for the diagram of
functors
\begin{equation}\label{E.1.3}
\xymatrix{ FU_T  \phi_T U_T =FT U_T \ar@{->}@<0.5ex>[rr]^-{FU_T
\varepsilon_T} \ar@ {->}@<-0.5ex> [rr]_-{\beta U_T}&&
FU_T,}\end{equation}
where $\varepsilon_T : \phi_T U_T  \to 1$ is the
counit of $\phi_T \dashv U_T$,  then $R$ is left
adjoint to $\overline{R}: \B \to \A_T$. It is easy to see that for
any $(a, h_a) \in \A_T$, the $(a, h_a)$-component in the diagram \ref{E.1.3} is
the pair
\begin{equation}\xymatrix{FT(a)\ar@{->}@<0.5ex>[rr]^-{F(h_a)} \ar@
{->}@<-0.5ex> [rr]_-{\beta_a}& & F(a)}
\end{equation}
which is a reflexive pair since
$\beta_a \cdot F(e_a)=F(h_a) \cdot F(e_a)=1$.
Thus we have:
\smallskip

{\em If $\B$ admits coequalisers of reflexive pairs, then the
functor $\overline{R}$ admits a left adjoint.}
\end{thm}

\medskip

 So far we have dealt with (co)module structures on functors.
It is also of interest to consider the corresponding relations between monads and comonads.

\begin{thm}\label{Gal-mon}{\bf Definitions.} \em
Let $\textbf{T}=(T, m,e)$ be a monad and $\textbf{G}=(G,
\delta,\varepsilon)$ a comonad on $\A$.
We say that
$\textbf{G}$ is {\em $\textbf{T}$-Galois}, if there exists a left
$\textbf{T}$-module structure $\alpha: TG \to G$ on the functor $G$
such that the composite $$\xymatrix{\gamma^G : TG \ar[r]^-{T
\delta}& TGG \ar[r]^-{\alpha G}& GG}$$ is an isomorphism.

Dually,
$\textbf{T}$ is {\em $\textbf{G}$-Galois}, if there is a left
$\textbf{G}$-comodule structure $\alpha: T \to GT$ on the functor
$T$ such that the composite $$\xymatrix{\gamma_T : TT
\ar[r]^-{\alpha T }& GTT \ar[r]^-{G m}& GT}$$ is an isomorphism.
\end{thm}

We need the following (dual of \cite[Lemma 21.1.5]{S})
\begin{proposition}\label{P.3.1-N}
Let $\eta, \varepsilon :F \dashv R : \C \to \A $ and
$\eta ' , \varepsilon' :F' \dashv R' : \C \to \B $ be adjunctions
and let $$ \xymatrix{\A \ar[rd]_{F} \ar[r]^{X}&
\B \ar[d]^{F'}\\
 &  \C}
$$ be a diagram of categories and functors with $F'X=F$. Write $\alpha$ for the composition
$$\xymatrix{XR \ar[r]^-{\eta' XR}& R'F'XR=R'FR \ar[r]^-{R'
\varepsilon}& R'\, .}$$ Then the natural transformation $S_X=F'
\alpha: FR=F'XR \to F'R'$ is a morphism of comonads.
\end{proposition}

Note that for the commutative diagram (see \ref{com-fun})
$$
\xymatrix{ \B  \ar[r]^-{\overline{F}}\ar[rd]_{F} & \A^{G} \ar[d]^{U^{{G}}}\\
& \A\, ,}$$ where $F$ has a right adjoint $R$, the related comonad
morphism $S_{\overline{F}}: FR \to G$ is just the comonad morphism
$t_{\overline{F}}: FR \to G.$
\smallskip

It is shown in \cite{S} that:

\begin{proposition}\label{P.3.2-N}
Let $ F \dashv R : \D \to \A $,
$ F' \dashv R' : \D \to \B $ and
$ F'' \dashv R'' : \D \to \C $ be adjunctions and let
$$ \xymatrix{\A \ar[rd]_{F} \ar[r]^{X}&
\B \ar[d]^{F'}\ar[r]^{Y}&
\C  \ar[ld]^{F''}\\
 &  \D&}
$$ be a diagram of categories and functors with $F'X=F$ and $F''Y=F'$.
Write $S_X$ for the
comonad morphism $FR \to F'R'$, $S_Y $ for the comonad morphism
$F'R' \to F'' R''$ and $S_{YX}$ for the comonad morphism $FR \to
F''R''$ that exist according to the previous proposition. Then
$S_{YX}=S_Y S_X.$
\end{proposition}

\section{Entwinings}

We fix a mixed distributive law (entwining) $\lambda : TG \to GT$ from the
monad $\textbf{T}=(T,m,e)$ to the comonad $\textbf{G}=(G,\delta,\ve)$,
and write
$\widehat{\textbf{T}}=(\widehat{T}, \widehat{m},\widehat{e})$
for the monad on $\A^G$ lifting $\textbf{T}$, and
$\widehat{\textbf{G}}=(\widehat{G},\widehat{\delta},\widehat{\ve})$
for the comonad on $\A_T$
lifting $\textbf{G}$ (e.g. \cite[Section 5]{W}).

  It is well-known that for any object $(a,h_a)$ of $\A_T$,
$$\bullet\; \widehat{G}(a, h_a)=(G(a), G(h_a) \cdot \lambda_a) ,
\quad \bullet \;
(\widehat{\delta})_{(a, h_a)}=\delta_a,
\quad \bullet \;
(\widehat{\varepsilon})_{(a, h_a)}=\varepsilon_a ,$$
while for any object $(a, \theta_a)$ of the category $\A^{G}$,

$$\bullet\; \widehat{T}(a, \theta_a)=(T(a), \lambda_a \cdot
T(\theta_a));
\quad \bullet\; (\widehat{m})_{(a, \theta_a)}=m_a,
\quad \bullet \;(\widehat{e})_{(a, \theta_a)}=e_a,$$
and that one has an isomorphism of categories
$$(\A^{G})_{\widehat{T}} \simeq (\A _{T})^{\widehat{G}}.$$ We
write $\A^{G}_{T}(\lambda)$ (or just $\A^{G}_{T},$ when the mixed
distributive law $\lambda$ is understood) for the category whose
object are triple $(a, h_a, \theta_a)$, where $(a, h_a) \in
\A_{T}$ and $(a,\theta_a) \in \A^{G}$ with commuting diagram

\begin{equation}\label{D.1.1}
\xymatrix{
T(a) \ar[r]^-{h_a} \ar[d]_-{T(\theta_a)}& a \ar[r]^-{\theta_a}& G(a) \\
TG(a) \ar[rr]_-{\lambda_a}&& GT(a). \ar[u]_-{G(h_a)}}
\end{equation}

Let $K: \A \to (\A^G)_{\widehat{T}}$ be a functor inducing a commutative
diagram
\begin{equation}\label{E.2.1}
\xymatrix{ \A  \ar[r]^-{K}\ar[rd]_{\phi^G} & (\A^G)_{\widehat{T}} \ar[d]^{U_{\widehat{T}}}\\
& \A^G .}\end{equation}
 Write $\alpha_K : \widehat{T}\phi^G \to
\phi^G$ for the corresponding  $\widehat{\textbf{T}}$-module structure
on $\phi^G$ (see \ref{T-mod}). Since $\widehat{\textbf{T}}$ is the lifting of
$\textbf{T}$ corresponding to $\lambda$, $U^G \widehat{{T}}=T U^G $
and one has the natural transformation
$$\alpha= U^G (\alpha_K): U^G
\widehat{{T}} \phi^G =TU^G \phi^G=TG \lra U^G\phi^G =G.$$
It is easy to see that $\alpha$ provides a left
$\textbf{T}$-module structure on $G$ with commutative diagram
\begin{equation}\label{E.2.2}
 \xymatrix{ TG \ar[r]^-{\alpha } \ar[d]_{T\delta}
& G
\ar[r]^-{\delta} &GG \\
TGG \ar[rr]_-{\lambda G} && GTG \ar[u]_{G \alpha }.}
\end{equation}

 Conversely, a natural transformation
$$\alpha: U^G
\widehat{{T}} \phi^G =TU^G \phi^G=TG \lra U^G\phi^G =G$$
 making $G$ a left $\textbf{T}$-module,  can be lifted to a left
$\widehat{\textbf{T}}$-module structure on $\phi^G$ if and only if  for every $a
\in \A$, $\alpha_a : TG(a) \to G(a)$ is a morphism in $\A^G$ from
the $\textbf{G}$-coalgebra $(GT(a), \lambda_{G(a)}\cdot
T(\delta_a))$ to the $\textbf{G}$-coalgebra $(G(a), \delta_a)$,
which is just to say that the $a$-component of the diagram (\ref{E.2.2}) commutes.
Thus we have proved:

\begin{proposition} The assignment $$(K: \A \to (\A^G)_{\widehat{T}})
\longmapsto (U^G (\alpha_K): TG \to G)$$ yields a bijection
between functors $K$ making the diagram (\ref{E.2.1}) commute and left
$\textbf{T}$-module structures $\alpha: TG \to G$ on $G$ for which
the diagram (\ref{E.2.2}) commutes.
\end{proposition}

Now let $K': \A \to (\A_T)^{\widehat{G}}$ be a functor inducing a commutative
diagram
\begin{equation}\label{E.2.1b}
\xymatrix{ \A  \ar[r]^-{K'}\ar[rd]_{\phi_T} & (\A_T)^{\widehat{G}} \ar[d]^{U^{\widehat{G}}.}\\
& A_T}\end{equation}
 Write $\beta_{K'} :  \phi_T \to \widehat{G}\phi_T$
for the corresponding  $\widehat{\textbf{G}}$-comodule structure
on $\phi_T$ (see \ref{com-fun}).
One has the natural transformation
$$\beta=(U_T (\beta_{K'}): U_T \phi_T=T \to U_T \widehat{ G}
\phi_T=G U_T \phi_T=GT $$
which induces a $\bG$-comodule structure on $T$ with commutative diagram
\begin{equation}\label{D.beta}
\xymatrix{ TT \ar[r]^-{m } \ar[d]_{T \beta} & T
\ar[r]^-{\beta} &GT \\
TGT \ar[rr]_-{\lambda T} && GTT \ar[u]_{G m}}
\end{equation}

 {}From this we obtain:

\begin{proposition}\label{P.2.2} The assignment
$$(K': A \to (\A_{T})^{\widehat{G}})
\longmapsto (U_T (\beta_{K'}):T \to GT)$$
yields a bijection between functors $K'$ making the diagram (\ref{E.2.1b})
commute
and left $\textbf{G}$-comodule structures
$\beta: T \to GT$ on the functor $T$
for which the diagram (\ref{D.beta}) commutes.
\end{proposition}

We know from \cite{G} that to give a functor  $K': \A \to
(\A_T)^{\widehat{G}}$ making the diagram (\ref{E.2.1b})
commute is to give a natural transformation $\alpha: U_T
\to U_T \widehat{G}$ making $U_T$  a right
$\widehat{G}$-comodule. For any $(a, h_a) \in \A_T$,
$\widehat{G}(a, h_a)=(G(a), G(h_a) \cdot \lambda_a)$, the $(a,
h_a)$-component $\alpha_{(a, h_a)}$ is a morphism $a \to G(a)$ in
$\A$ with commutative diagrams
$$
\xymatrix{ a \ar@{=}[dr] \ar[r]^-{\alpha_{(a, h_a)}}& G(a)
\ar[d]^{\varepsilon_a} &  & a\ar[d]_{\alpha_{(a, h_a)}}
\ar[rr]^-{\alpha_{(a, h_a)}}&&
G(a) \ar[d]^{\delta_{G(a)}}\\
& a, & & G(a) \ar[rr]_{G(\alpha_{(a, h_a)})}&& GG(a),}
$$
and the corresponding comonad morphism
$t_K : \phi_T U_T \to \widehat{G} $ is the composite
$$
\xymatrix{\phi_T U_T  \ar[r]^-{\phi_T \alpha}& \phi_T U_T
\widehat{G} \ar[r]^-{\varepsilon_T \widehat{G}}&\widehat{G}\, .}$$
Then, since for any $(a, h_a) \in \A_T$,
$(\varepsilon_T)_{(a,h_a)}=h_a$, the component $(t_K)_{(a,h_a)}$
is the composite
$$
\xymatrix{T(a) \ar[rr]^-{T(\alpha_{(a, h_a)})}&& TG(a)
\ar[r]^-{\lambda_a}& GT(a) \ar[r]^-{G(h_a)}& G(a)\, .}$$

Now it follows from Proposition \ref{P.1.7}:

\begin{theorem}\label{T.2.1-ent}
In the situation described above, the functor $K'$ is
an equivalence of categories if and only if for any $(a, h_a) \in \A_T$, the
composite $G(h_a) \cdot \lambda_a \cdot T(\alpha_{(a, h_a)})$ is an
isomorphism and the functor $\phi_T$ is comonadic.
\end{theorem}

For the dual situation, let $K: \A \to (\A^G)_{\widehat{T}}$
be a functor inducing commutativity of the diagram (\ref{E.2.1}).
Since the functor $\phi^G$ has a left adjoint
 $U^G : \A^G \to \A$, it follows from \cite{G} that to give
such a functor is to give a right $\widehat{T}$-module structure
$\alpha : U^G \widehat{T} \to U^G$ on $U^G$.

For any $(a, \theta_a) \in \A^G$, $\widehat{T}(a,
\theta_a)=(T(a),  \lambda_a \cdot T(\theta_a))$, the $(a,
\theta_a)$-component $\alpha_{(a, \theta_a)}$ is a morphism $T(a)
\to a$ in $\A$ with commutative diagrams
$$
\xymatrix{ a \ar@{=}[dr] \ar[r]^-{e_a}& T(a) \ar[d]^{\alpha_{(a,
\theta_a)}} &\quad& TT(a)\ar[d]_{m_a} \ar[rr]^-{T(\alpha_{(a,
\theta_a))}}&&
T(a) \ar[d]^{\alpha_{(a, \theta_a)}}\\
& a, && T(a) \ar[rr]_{\alpha_{(a, h_a)}}&& a,}
$$
 and the corresponding monad morphism
$t_K : \phi^G U^G \to \widehat{T} $ is the composite
$$
\xymatrix{\widehat{T}  \ar[r]^-{\eta^G \widehat{T}}& \phi^G U^G
\widehat{T} \ar[r]^-{\phi^G \alpha}& \phi^GU^G \, .}$$
Now, since for any $(a, \theta_a) \in \A^G$,
$(\eta^G)_{(a,\theta_a)}=\theta_a$, the component $(t_K)_{(a,\theta_a)}$ is the
composite
$$
\xymatrix{T(a) \ar[rr]^-{T(\theta_a)}&& TG(a) \ar[r]^-{\lambda_a}&
GT(a) \ar[rr]^-{G(\alpha_{(a, h_a)})}& & G(a)\, .}$$

As a consequence we get from Proposition \ref{P.1.2}:

\begin{theorem}\label{T.2.2-ent}
In the situation described above, the functor $K$ is
an equivalence of categories if and only if for any $(a, \theta_a) \in \A^G$,
the composite $G(\alpha_{(a, h_a)}) \cdot \lambda_a \cdot
T(\theta_a)$ is an isomorphism and the functor $\phi^G$ is monadic.
\end{theorem}

The following observation is probably known but we
are not aware of a suitable reference.
Recall that a functor $i: \C \to \A$
with $\C$ a small category is {\em dense}, if the functor
$$\widetilde{i}:
A^{op} \to [\C, \Set],\, a \to \Mor_\A(i(-), a),$$
is full and faithful.

\begin{lemma}\label{L.2.3}
Let $i: \C \to \A$ be a dense functor. Given two adjunctions $$F
\dashv U ,\, F' \dashv U': \A \to \B,$$ and a natural
transformation $\tau: F \to F'$, then $\tau$ is an isomorphism of
functors if and only if $\tau i : F i \to F'i$ is so.
\end{lemma}
\begin{proof} Write $\tau' : U' \to U$ for the natural
transformation corresponding to $\tau$, that is $\tau$ and $\tau'$ are mates,
denoted by $\tau \dashv \tau'$
(e.g. \cite[7.1]{MW}, \cite[2.2]{BBW}).
 Then $\tau$ is an
isomorphism if and only if  $\tau'$ is so. So it is enough to show that $\tau'$
is an isomorphism. Since $\tau \dashv \tau'$, the diagram
$$
\xymatrix{ \Mor_\B(F'i(a), b) \ar[rr]^-{\alpha'_{i(a), \,b}}
\ar[d]_{\Mor_\B(\tau_{i(a)}, \,b)} && \Mor_\A(i(a),U'( b)) \ar[d]^{\Mor_\A(i(a),
\tau'_{b})}\\
\Mor_\B(Fi(a), b) \ar[rr]_{\alpha_{i(a), \,b}}&& \Mor_\A(i(a), U(b)),}
$$where $\alpha$ (resp. $\alpha'$) is the bijection corresponding to
the adjunction $F \dashv U$ (resp. $F' \dashv U'$), commutes for all
$a, b \in \A$. Since $\tau_{i(a)}$ is an isomorphism by our
assumption on $\tau$, it follows that the natural transformation
$\Mor_\A(i(a), \tau'_{b})$ is an isomorphism, implying - since $i$ is
dense - that $\tau' : U' \to U$ is an isomorphism.
\end{proof}

\begin{proposition}\label{P.2.4}
Suppose $K': \A \to (\A_T)^{\widehat{G}}$ to be a functor
with $ U^{\widehat{G}}K'=\phi_T$ and
let $\beta_K : \phi_T \to \widehat{G}\phi_T$
be the corresponding  $\widehat{\textbf{G}}$-comodule
structure on $\phi_T$ (see \ref{com-fun}).
Suppose that
\begin{rlist}
\item $\A$ admits equalisers of coreflexive pairs and
both $T$ and $G$ have right adjoints, or
\item $\A$ admits small colimits and both $T$ and $G$ preserve them.
\end{rlist}
 Then $(\phi_T,\beta_K)$ is $\widehat{\textbf{G}}$-Galois if and only if
$(T, U_T(\beta_K))$ is $\textbf{G}$-Galois.
\end{proposition}
\begin{proof} For any $a \in \A$, the $\phi_T(a)=(T(a),m_a)$-component
 of $t_K : \phi_T U_T \to \widehat{\textbf{G}}$ is just
$(\gamma_T)_a$ (see \ref{Gal-mon}). Thus it is enough to show that $t_K$ is an
isomorphism if and only if its restriction to free $\textbf{T}$-modules is.

(i)  If $T$ has a right adjoint, there
exists a comonad $\textbf{H}$ inducing an
isomorphism of categories $\A_T \simeq \A^H$; this implies that the
functor $U_T$ is comonadic and hence has a right adjoint. It
follows that the composite $G U_T$ also has a right adjoint. Next,
since $\widehat{\textbf{G}}$ is the lifting of $\textbf{G}$, we
have the commutative diagram
$$\xymatrix{ \A_T \ar[r]^-{\widehat{G}} \ar[d]_{U_T} & \A_T \ar[d]^{U_T}\\
\A \ar[r]_{G}& \A.}$$
 Since
\begin{itemize}
    \item $G U_T$ has a right adjoint,
    \item the functor $U_T$ is
comonadic, and
    \item $\A_T$ admits equalisers of coreflexive pairs since $\A$
    does so,
\end{itemize} it follows from the dual of \cite[Theorem A.1]{D} that the functor $\widehat{G}$
has a right adjoint.

Now, since the full subcategory of $\A_T$ given by free
$\textbf{T}$-modules is dense in $\A_T$, it follows from Lemma \ref{L.2.3}
that $t_K : \phi_T U_T \to \widehat{G}$ is an isomorphism if and only if its
restriction to free $\textbf{T}$-modules is.
\smallskip

(ii)  Since $T$ preserves colimits,  the category $\A_{T}$ admits
colimits and the functor $U_{T} : \A_{T} \to \A$ creates them. Thus
\begin{itemize}
\item the functor $\phi_T U_T$ preserves colimits;
\item any functor $L: \B \to \A_{T}$ preserves colimits if
and only if the composite $U_{T}L $ does; so, in particular, the
functor $\widehat{T}$ preserves colimits, since $U_{T} \widehat{T}=T
U_{T}$ and $T U_{T}$ is the composite of two
colimit-preserving functors.
\end{itemize}
The full subcategory of $\A_{T}$ given by the free
$\textbf{T}$-modules is dense and since the functors $\phi_T U_{T}$
and $\widehat{T}$ both preserve colimits, it follows from
\cite[Theorem 17.2.7]{S} that the natural transformation
$$t_{K}:\phi_T U_{T}\to \widehat{T}$$
is an isomorphism if and only if its restriction to the free
$\textbf{T}$-modules is so; i.e. if $(t_{K})_{\phi_T(a)}$ is an
isomorphism for all $a \in \A$.
This completes the proof.
\end{proof}

Dually, one has

\begin{proposition}\label{P.2.5}
Suppose that $K : \A \to (\A^G)_{\widehat{T}}$ is a
functor with $U_{\widehat{T}} K=\phi^G$ and let $\alpha_K :
\widehat{T}\phi^G \to \phi^G$ be the corresponding
$\widehat{\textbf{T}}$-module structure on $\phi^G$. Suppose that
\begin{rlist}
\item $\A$ admits coequalisers of reflexive pairs and both $T$ and $G$
have left adjoints, or
\item
$\A$ admits all small limits
and both $T$ and $G$ preserve them.
\end{rlist}
Then $(\phi^G, \alpha_K)$ is
$\widehat{\textbf{T}}$-Galois if and only if $(G, U^G(\alpha_K))$ is
$\textbf{T}$-Galois.
\end{proposition}

The results of the preceding two propositions may be compared with
B\"ohm and Menini's \cite[Theorem 3.3]{BoMe}.

\section{Grouplike morphisms}

In  this section we extend the theory of Galois corings $\mathcal C$
over a ring $A$ to entwinings of a monad $F$ and a comonad $G$ on general categories.
For this we extend the notion of a grouplike element in $\mathcal C$
(e.g. \cite[28.1]{BW})
to the notion of a grouplike natural transformation $I\to G$.

\begin{thm}\label{grouplike}{\bf Definition.} \em
Let $\bG=(G,\delta,\ve)$ be a comonad on a
category $\A$.
A natural transformation $g:I\to G$ is called a {\em grouplike
morphism} provided it induces commutative diagrams
$$\xymatrix{ I \ar[r]^g \ar[dr]_= & G \ar[d]^\ve \\
                & I ,}  \quad
  \xymatrix{ I \ar[r]^g \ar[dr]_{gg} & G \ar[d]^\delta \\
                & GG .} $$
\end{thm}

Comonads with grouplike morphisms are called {\em computational}
in \cite{BroGev} (see also \cite{Mul}).
The next result transfers Proposition 5.1 in \cite{M}.

\begin{thm}\label{group.com}{\bf Grouplike morphisms and comodule structure.}
Let $\bF=(F,m,e)$ be a monad and $\bG=(G,\delta,\ve)$ a comonad on a
category $\A$ with an entwining  $\lambda:FG\to GF$.
If $G$ has a grouplike morphism $g:I\to G$, then $F$ has two left
$G$-comodule structures (see \ref{com-fun}) given by
$$\xymatrix{ (1)\quad\tilde g: F \ar[r]^{Fg}& FG \ar[r]^\lambda & GF }
\; \mbox{ and } \;(2)\quad  gF:F\to GF  .$$
\end{thm}
\begin{proof} (1) In the diagram
$$\xymatrix{F \ar[r]^{Fg} \ar[rd]_= & FG \ar[d]^{F\ve}  \ar[r]^\lambda &
                                                 GF \ar[d]^{\ve F} \\
    & F \ar[r]_= & F ,} $$
the triangle is commutative by the grouplike properties of $g$ and
the square is commutative by the properties of the entwining
$\lambda$.
In the diagram
$$\xymatrix{  F \ar[r]^{Fg}\ar[d]_{Fg}\ar[rd]^{Fgg}
 & FG \ar[r]^\lambda \ar[d]^{F\delta}  & GF \ar[dd]^{ \delta F} \\
    FG \ar[d]_\lambda & FGG \ar[d]^{\lambda G}   &   \\
    GF \ar[r]^{GFg} & GFG \ar[r]^{G\lambda} & GGF, }
$$
the right rectangle is commutative by properties of entwinings, the
triangle is commutative by properties of the grouplike morphism $g$, and the
pentagon is commutative by naturality of composition.

This shows that $\tilde g$ makes $F$  a left $G$-comodule.
\smallskip

(2) To say that $(F, gF:F\to GF )$ is a left $\bG$-comodule
is to say that the  diagrams
$$\xymatrix{ F \ar[r]^{gF} \ar[dr]_= & GF \ar[d]^{\ve F} \\
                & F ,}  \quad
  \xymatrix{ F \ar[r]^{gF} \ar[d]_{gF} & GF \ar[d]^{\delta F}\\
       GF \ar[r]_{GgF}         & GGF .} $$
are commutative. Using the fact that $$GgF \cdot gF = ggF,$$
the commutativity of these diagrams follows
from the definition of a grouplike morphism.
\end{proof}

The pattern of the proof of \cite[Proposition 5.3]{M} also yields:

\begin{thm}\label{F.mixed} {\bf $F$ as mixed bimodule.}
With the data given in \ref{group.com},
$(F,m, \tilde g)$ is a mixed $(F,G)$-bimodule.
\end{thm}
\begin{proof}
We need to show commutativity of the diagram
$$\xymatrix{FF\ar[r]^m \ar[d]_{F\tilde g}  & F \ar[r]^{\tilde g}& GF \\
   FGF \ar[rr]^{\lambda F} & & GFF \ar[u]_{Gm} .} $$
However, by the definition of $\tilde g$, we get the diagram
$$\xymatrix{FF\ar[r]^m \ar[d]_{FFg}  & F \ar[r]^{Fg}& FG \ar[r]^{\lambda}& GF \\
  FFG \ar[r]_{F\lambda} \ar[rru]^{mG}
& FGF \ar[rr]^{\lambda F} & & GFF \ar[u]_{Gm} ,} $$ in which the
right pentagon is commutative since $\lambda$ is an entwining and
the triangle is commutative by naturality of composition.
This proves our claim.
\end{proof}

  Combining  \ref{P.2.2}, \ref{group.com} and \ref{F.mixed}
yields the existence of a functor
$K_g : \A \to (\A_F)^{\widehat{G}}$ making the diagram
$$\xymatrix{
\A \ar[r]^-{K_g} \ar[dr]_{\phi_F}& (\A_F)^{\widehat{G}}
\ar[d]^{U^{\widehat{G}}}\\
&\A_F }$$ commute. Note that $K_g(a)=((F(a), m_a), \tilde{g}_a)$.

Now assume that $\A$ admits equalisers. Then the category of
endofunctors of $\A$ also has equalisers and we have the

\begin{thm}\label{equal-func}{\bf Equaliser functor.}
With the data given in \ref{group.com},
define a functor $F^g$ as an equaliser of functors

$$\xymatrix{
F^g\ar[r]^{i_F} & F \ar[rr]^{gF} \ar[dr]_{Fg}& & GF \\
           & & FG \ar[ru]_\lambda . } $$
Then $F^g$ is a monad on $\A$ and $i_F:F^g\to F$ is a monad
morphism.
\end{thm}

\begin{proof} We adapt the proof of \cite[5.2]{M}.
The following two diagrams are commutative by naturality of
composition,
$$\xymatrix{ I \ar[r]^e \ar[d]_g & F \ar[d]^{gF} \\
              G \ar[r]_{Ge}  & GF,} \quad
  \xymatrix{ I \ar[r]^e \ar[d]_{g} & F \ar[d]^{Fg} \\
                G \ar[r]_{eG}& FG .} $$
 Since $\lambda \cdot eG=Ge$, it follows that
$$\lambda \cdot
Fg \cdot e=\lambda \cdot eG \cdot g=Ge \cdot g=gF \cdot e.$$ Thus
there exists a unique morphism $e': I \to F^g$ yielding a
commutative diagram
$$
\xymatrix{F^g \ar[r]^{i_F} &F\\
I  \ar[u]^{e'} \ar[ru]_{e}&}$$

Observe that
\begin{itemize}
\item [($\alpha$)] the   diagrams
$\xymatrix{FF \ar[r]^{FFg} \ar[d]_{m}& FFG \ar[d]^{mG}\\
F \ar[r]_{Fg}& FG,} \qquad
\xymatrix{ FF \ar[r]^{gFF} \ar[d]_{m}& GFF \ar[d]^{Gm}\\
F \ar[r]_{gF}& GF}$

commute by naturality of composition,

\item [($\beta$)]
$\lambda \cdot mG=Gm \cdot \lambda F \cdot F \lambda$,
    since   $\lambda$  is an entwining;

\item [($\gamma$)]
$\lambda \cdot Fg \cdot i_F=gF \cdot i_F$, since  $i_F$ is an
equaliser of $gF$ and $\lambda \cdot Fg$,

\item [($\delta$)]
$i_F i_F=i_F F \cdot F^g i_F=F i_F \cdot i_F F^g$,  by naturality of
composition.
\end{itemize}
 Hence we have
$$\begin{array}{rcl}
\lambda \cdot Fg \cdot m \cdot i_Fi_F&  =_{(\alpha)} &
  \lambda \cdot mG \cdot FFg \cdot i_Fi_F \\
  &=_{(\beta)}  &
   Gm \cdot \lambda F \cdot F \lambda \cdot FFg \cdot i_Fi_F\\
 &=_{(\delta)}
  & Gm \cdot \lambda F \cdot F \lambda \cdot FFg \cdot F i_F \cdot i_F
F^g \\
  &=_{(\gamma) } &
 Gm \cdot \lambda F \cdot  FgF \cdot F i_F \cdot i_F F^g \\
  &=_{(\delta) }&
Gm \cdot \lambda F \cdot  FgF \cdot i_F F \cdot F^g i_F \\
  &=_{(\gamma) }&
 Gm \cdot   gFF \cdot i_F F \cdot F^g i_F \\
  &=_{(\delta)} &
 Gm \cdot   gFF \cdot i_Fi_F \\
  &=_{(\alpha)} & gF \cdot m \cdot i_Fi_F.
\end{array}$$
Considering now the diagram
$$
\xymatrix{F^g F^g \ar@{-->}[dd]_{m'}\ar[r]^{i_{_F}i_{_F}} &FF
\ar[dd]_{m} \ar@/^2pc/@{->}[rrrr]^{gFF } \ar[rr]^{FgF} && FGF
\ar[rr]^{\lambda
F} && GFF \ar[dd]^{Gm}\\\\
F^g \ar[r]_{i_F}&F \ar@/^2pc/@{->}[rrrr]^{gF } \ar[rr]^{Fg} && FG
\ar[rr]^{\lambda } && GF \, ,}
 $$
one sees that there exists a unique
morphism $m': F^gF^g \to F^g$ making the left square of the diagram
commute. The result now follows from \cite[Lemma 3.2]{B}.
\end{proof}

  As we have seen, the morphism $\beta=\tilde{g}: F \to GF$
makes $F$ a left $G$-comodule. Consider the related
functor $K_g : \A \to (\A_F)^{\widehat{G}}$ and write $t: \phi_F U_F
\to \widehat{G}$ for the corresponding morphism of comonads on
$\A_F$. It is easy to see that for any $(a, h_a) \in \A_F$, $t_{(a,
h_a)}$ is the composite
\begin{equation}\label{t-comp}
\xymatrix{F(a) \ar[r]^-{F(g_a)}& FG(a) \ar[r]^-{\lambda_a}& GF(a)
\ar[r]^-{G(h_a)}& G(a).}
\end{equation}
Since $Fg \cdot e=eG \cdot g$ by
naturality of composition and $\lambda \cdot eG =Ge$, the
$(a, h_a)\in \A_F$-component of the morphism
$$
\xymatrix{\gamma : U_F \ar[r]^-{\eta^F U_F}& U_F\phi_FU_F
\ar[r]^-{U_F t}& U_F \widehat{G}}$$ is just the morphism $g_a : a
\to G(a).$ It follows that the monad generated by the functor $K_g$
and its right adjoint $R_g$ is given by the equaliser of the diagram
$$
\xymatrix{F \ar@{->}@<0.5ex>[r]^-{gF} \ar@ {->}@<-0.5ex>
[r]_{\tilde{g}=\lambda\cdot Fg}& GF.}$$
Thus $F^g$ is just the monad on $\A$
generated by the adjunction $K_g \dashv R_g$.

Since any functor with a right adjoint is full and faithful  if
and only if the
unit of the adjunction is an isomorphism, we have the

\begin{proposition}\label{P.3.6.x} Let $g: 1 \to G$ be a grouplike morphism. Then
the corresponding functor $K_g : \A \to (\A_F)^{\widehat{G}}$ is
full and faithful if and only if the functor $F^g$ is (isomorphic to) the
identity monad on $\A$.
\end{proposition}
\smallskip

For an entwining $\lambda : TG \to GT$ and a grouplike morphism $g:
1 \to G$, for any $(a, h_a) \in \A_F$, the ${(a, h_a)}$-component $t_{(a,h_a)}$
 of the comonad morphism $t: \phi_F U_F \to \widehat{G}$,
corresponding to the functor $K_g : \A \to (\A_F)^{\widehat{G}}$, is
given in (\ref{t-comp}).
Consider the diagram

$$\xymatrix{F(a) \ar@{}[rrdd]^{(1)} \ar[rr]^-{F(e_a)}\ar[dd]_-{F(g_a)}&&
FF(a) \ar@{}[rrdd]^{(2)} \ar[dd]^-{F(g_{F(a)})} && F(a)
\ar[ll]_-{e_{F(a)}}
\ar[dd]^-{g_{F(a)}}\\\\
FG(a) \ar@{}[rrdd]^{(3)} \ar[rr]^-{FG(e_a)} \ar[dd]_-{\lambda_a} &&
FGF(a) \ar@{}[rdd]^{(4)}\ar[dd]_-{\lambda_{F(a)}}&& GF(a)
\ar[ll]_-{e_{GF(a)}}
\ar[lldd]^-{G(e_{F(a)})}\ar@{=}[dddd]\\\\
GF(a) \ar@{=}[rrdd]\ar[rr]^-{GF(e_a)} && GFF(a)\ar@{}[rrdd]^{(6)} \ar@{}[lld]^{(5)}\ar[dd]^{G(m_a)}&& \\\\
&& GF(a) \ar@{=}[rr]&& GF(a)\,,}$$ in which

\begin{itemize}
    \item diagram (1) is commutative by naturality of $g: 1 \to G$;
    \item diagram (2) is commutative by naturality of composition;
    \item diagram (3) is commutative by naturality of $\lambda:FG\to GF$;
    \item diagram (4) is commutative since $\lambda$ is an entwining, and
    \item diagrams (5) and (6) are commutative since $F$ is a monad.
\end{itemize}
It follows from the commutativity of this diagram that
the diagram
\begin{equation}\label{D.3.1-ent}
\xymatrix{F(a)\ar@{=}[ddd]
\ar@{->}@<0.5ex>[rr]^-{e_{F(a)}}
\ar@ {->}@<-0.5ex> [rr]_{F(e_a)}&&FF(a) \ar[d]^{F(g_{F(a)})}\\
&&FGF(a) \ar[d]^{\lambda_{F(a)}}\\
&& GFF(a) \ar[d]^{G(m_a)}\\
F(a) \ar@{->}@<0.5ex>[rr]^-{g_{F(a)}} \ar@ {->}@<-0.5ex>
[rr]_{\lambda_a \cdot F(g_a)}&&GF(a)}
\end{equation}
is serially commutative.

\begin{proposition}\label{P.3.5-ent}
Let $\lambda : TG \to GT$ be an entwining and  $g:
1 \to G$ be a grouplike morphism. If the monad $F$ is of descent
type (that is, the free $F$-algebra functor $\phi_F : \A \to
\A_F$ is precomonadic) and if the monad $F$ is $G$-Galois w.r.t. the
$G$-coaction $\widetilde{g}: F \to GF$ (see \ref{group.com}),
then the monad $F^g$ is (isomorphic to) the identity monad.
\end{proposition}
\begin{proof} To say that $F$ is of descent type is to say that the
diagram
$$\xymatrix{a \ar[r]^-{e_a}&F(a) \ar@{->}@<0.5ex>[rr]^-{e_{F(a)}}
\ar@ {->}@<-0.5ex> [rr]_{F(e_a)}&&FF(a)}$$ is a coequaliser diagram
for all $a \in \A$ (see \cite{BW}), while to say that the monad $F$
is $G$-Galois w.r.t $G$-coaction $\widetilde{g}: F \to GF$ is to say
that, for any $a \in \A$, the composite $G(m_a) \cdot \lambda_{F(a)}
\cdot F(g_{F(a)})$ is an isomorphism. The result now follows from
the commutativity of the diagram (\ref{D.3.1-ent}).
\end{proof}

\begin{thm}\label{adj.iF}{\bf Left adjoint of $(i_F)^*$.} \em
Since $i_F : F^g \to F$ is a morphism of monads, it induces a
functor
$$(i_F)^*: \A_F \to \A_{F^g},\quad (a, h_a) \mapsto (a, h_a \cdot (i_F)_a).$$ Moreover, when the
category $\A_F$ has coequalisers   of reflexive pairs  (which is
certainly the case if   $\A$ has coequalisers of reflexive pairs
and  $F$ preserves them), $(i_F)^*$ has a left adjoint
$(i_F)_! :\A_{F^g} \to \A_F$ which is defined as follows:
For notational reasons, write
\begin{center}
$\eta, \sigma :V \dashv U :\A_F \to \A$ (resp. $\eta', \sigma' : V'
\dashv U' : \A_{F^g} \to \A$)
\end{center}
for the forgetful-free adjunction
$(\phi_F,U_F)$ (resp. $(\phi_{F^g},U_{F^g})$.
Then $(i_F)_!$ is the
coequaliser of the diagram of functors and natural transformations
\begin{equation}\label{D.1.6}
\xymatrix{ VU'V'U' \ar@{->}@<0.5ex>[rr]^-{VU' \sigma'} \ar@
{->}@<-0.5ex> [rr]_{\beta}&& VU' \ar[rr]^-{q} && (i_F)_! \, ,}
\end{equation}
where $\beta$ is the   composite
$$\begin{array}{ll}
\xymatrix{ VU'V'U' \ar[rr]^-{VU'V'\eta U'}&& VU'V'UVU'=VU'V'U'(i_F)^*VU'} \\
\phantom{VU'V'U' VU'V'U'U'V'U'U'V'U'U'V'}\xymatrix{\ar[rr]^-{VU'\sigma'(i_F)^* VU'}&&VU'(i_F)^* VU'=VU VU'
 \ar[rr]^-{\sigma VU'}&& VU'\,.}
\end{array} $$
It is not hard to see that for any
$(a, h_a) \in \A_{F^g}$, the $(a, h_a)$-component of the diagram
(\ref{D.1.6}) is the diagram
$$
\xymatrix{FF^g(a) \ar@/^2pc/@{->}[rrrr]^{F(h_a) }
\ar[rr]_{F((i_F)_a)} && FF(a) \ar[rr]_{m_a} && F(a) \ar[r]^{q_a\quad\;}&
(i_F)_!(a,  h_a)}.$$
\end{thm}
\medskip

Let $\widehat{G}$ be the comonad on $\A_F$
that is the lifting of the comonad $\bG$
 corresponding to the entwining $\lambda$.
Then for any $(a, h_a)\in \A_F$,
$\widehat{G}(a, h_a)=(G(a), G(h_a) \cdot \lambda_a ).$

\begin{lemma}\label{L.3.8.x}
 For any $(a, h_a) \in \A_{F}$, the morphism $\tilde g_a: F(a) \to
GF(a)$ can be seen as a morphism in $\A_F$ from the free $F$-module
$V(a)=(F(a), m_a)$ to the $F$-module
$$\widehat{G}(V(a))=\widehat{G}(F(a), m_a)=(GF(a), G(m_a) \cdot
\lambda_{F(a)}).$$
\end{lemma}
\begin{proof} Consider the diagram
$$
\xymatrix{FF(a) \ar[rr]^{FF(g_a)} \ar[dd]_{m_a}&& FFG(a)
\ar[rr]^{F(\lambda_a)}\ar[dd]_{m_{G(a)}}&& FGF(a)
\ar[d]^{\lambda_{F(a)}}\\
&&&& GFF(a)\ar[d]^{G(m_a)}\\
F(a)\ar@{}[rruu]^{(1)} \ar[rr]_{F(g_a)} && FG(a)\ar@{}[rruu]^{(2)}
\ar[rr]_{\lambda_a}&&GF(a)\,,}$$
in which part (1) commutes by
naturality of $m$, while part (2) commutes since $\lambda$ is an
entwining. Thus the outer rectangle is commutative, which just means
that $\tilde g: F(a) \to GF(a)$ is a morphism in $\A_F$
from the free $F$-module $V(a)=(F(a), m_a)$ to the $F$-module
$(GF(a), G(m_a) \cdot \lambda_{F(a)}).$
\end{proof}

\begin{corollary} The collection $((\tilde g)_a)_{a \in \A}$ (see above) can be
seen as a natural transformation $\alpha_V : V \to \widehat{G}V$
making $V$ a left $\widehat{G}$-comodule.
\end{corollary}

\begin{proof} Using that for any $(a, h_a)\in \A_F$,
\begin{center}
 $\bullet\;$ $(\varepsilon_{\widehat{G}})_{(a,h_a)}=(\varepsilon_G)_a$,
 \quad $\bullet\;$ $(\delta_{\widehat{G}})_{(a,h_a)}=(\delta_G)_a$,
 \quad $\bullet\;$ $(F, \tilde g)$ is a left $G$-comodule,
\end{center}
it is not hard to prove that the pair $(V, \alpha_V)$ is a left
$\widehat{G}$-comodule.
\end{proof}
\begin{lemma}\label{coequal} With the notation above,
\begin{zlist}
\item the left rectangle in the diagram
$$
\xymatrix{ VU'V'U'\ar[d]_{\alpha_V U'V'U'}
\ar@{->}@<0.5ex>[rr]^-{VU' \sigma'} \ar@
{->}@<-0.5ex> [rr]_{\beta}&& VU' \ar[d]^{\alpha_V U'}
\ar[rr]^-{q} && (i_F)_! \ar@{-->}[d]^{\alpha_{(i_F)_!}}\\
\widehat{G}VU'V'U' \ar@{->}@<0.5ex>[rr]^-{\widehat{G}VU' \sigma'}
\ar@ {->}@<-0.5ex> [rr]_{\widehat{G}\beta}&&
\widehat{G}VU'\ar[rr]^-{\widehat{G}q} && \widehat{G}(i_F)_!}$$
 is serially commutative;
\item there exist a unique natural
transformation $\alpha_{(i_F)_!}: (i_F)_! \to \widehat{G}(i_F)_!$
making the right square of the diagram commute.
\end{zlist}
\end{lemma}
\begin{proof} (2) follows from the fact that  $q$ is a coequaliser of
$VU'\sigma'$ and $\beta$.
\smallskip

(1) To show that the left square is serially commutative, we
have to show that for any $(a, h_a)\in \A_F$, the diagram
$$
\xymatrix{FF^g(a)\ar[dd]_{(\tilde g)_{F^g(a)}}
\ar@/^2pc/@{->}[rrrr]^{F(h_a) } \ar[rr]_{F((i_F)_a)} && FF(a)
\ar[rr]_{m_a} && F(a) \ar[dd]^{(\tilde g)_a} \\\\
GFF^g(a) \ar@/^2pc/@{->}[rrrr]^{GF(h_a) } \ar[rr]_{GF((i_F)_a)} &&
GFF(a) \ar[rr]_{G(m_a)} && GF(a) }$$ is so. The left diagram below
is commutative by naturality of $g: I \to G$,
$$
\xymatrix{F^g(a) \ar[rr]^{h_a} \ar[d]_{g_{F^g\!(a)}}&&
a\ar[d]^{g_a}\\
GF^g(a) \ar[rr]_{G(h_a)}&& G(a),}\qquad
\xymatrix{F GF^g(a)
\ar[rr]^{FG(h_a)} \ar[d]_{\lambda_{F^g\!(a)}}&&
FG(a) \ar[d]^{\lambda_a}\\
GFF^g(a) \ar[rr]_{GF(h_a)}&& GF(a)}$$ while the right square is
commutative by naturality of $\lambda$. From this we obtain the
commutative diagram
$$
\xymatrix{FF^g(a) \ar[rr]^{F(h_a)} \ar[d]_{F(g_{F^g\!(a))}}&&
F(a)\ar[d]^{F(g_a)}\\
FGF^g(a)\ar[d]_{\lambda_{F^g\!(a)}} \ar@{-->}[rr]_{FG(h_a)}&& FG(a)\ar[d]^{\lambda_a}\\
GFF^g(a) \ar[rr]_{GF(h_a)}&& GF(a).}$$

Next, consider the diagram
$$
\xymatrix{FF^g(a) \ar[dddd]_{F(g_{F^g(a)})}
\ar[ddrr]^{F((i_F)_a)}\ar[rrrr]^{F((i_F)_a)}&&&&FF(a)\ar@{}[rrdd]^{(3)}\ar[dd]^{FF(g_a)}
\ar[rr]^{m_a}&&F(a)\ar[dd]^{F(g_a)}\\\\
&&FF(a)\ar@{}[rruu]^{(2)} \ar[rrdd]^{F(g_{F(a)})}&&
FFG(a)\ar[dd]^{F(\lambda_a)}
\ar[rr]_{m_{G(a)}}&& FG(a)\ar[dd]^{\lambda_{F(a)}}\\\\
FGF^g(a)\ar@{}[rrrrdd]_{(5)}\ar@{}[rruu]^{(1)}\ar[dd]_{\lambda_{F^g(a)}}
\ar[rrrr]_{FG((i_F)_a)}&&&&
FGF(a)\ar@{}[rr]^{(4)}\ar[dd]_{\lambda_{F(a)}} && GF(a)\\\\
GFF^g(a) \ar[rrrr]_{GF((i_F)_a)}&&&&GFF(a)\ar[rruu]_{G(m_a)}&&}$$ in
which
\begin{itemize}
    \item diagram (1) commutes by naturality of composition;
    \item diagram (2) commutes by ($\gamma$) in proof of \ref{equal-func};
    \item diagram (3) commutes by naturality of $m$;
    \item diagram (4) commutes since $\lambda$ is an entwining, and
    \item diagram (5) commutes by naturality of $\lambda$.
\end{itemize}
Thus the outer diagram is commutative and this
completes the proof of the lemma.
\end{proof}

\begin{thm}\label{nat-trans-S}{\bf Natural transformation $S_{\phi_{F^g}}$.}
\em
The pair $(V, \alpha_V)$ is a left $\widehat{G}$-comodule and by
commutativity of the diagram in \ref{coequal}, the pair $((i_F)_!,
\alpha_{(i_F)_!})$ is also a left $\widehat{G}$-comodule. Thus, as
noted in  \ref{com-fun}, there exists a unique functor
$\overline{i_F}: \A_{F^g} \to (\A_F)^{\widehat{G}}$ yielding
commutativity in the right triangle of the diagram

\begin{equation}\label{D.3.3-N}
\xymatrix{\A \ar[rrdd]_{\phi_F} \ar[rr]^{\phi_{F^g}}& &
\A_{F^g}\ar[dd]^{(i_F)_!} \ar[rr]^{\overline{i_F}} & &
(\A_F)^{\widehat{G}}  \ar[lldd]^{U^{\widehat{G}}}\\ \\
 & &  \A_F & &}
\end{equation}
where $U^{\widehat{G}} : (\A_F)^{\widehat{G}} \to \A_F$ is the
evident forgetful functor.

 A direct inspection shows that the diagram
$$\xymatrix{FF^g F^g  \ar@/^/@<+2.5ex>[rr]^{Fm'} \ar[r]^-{Fi_F F^g} & FFF^g \ar[r]^-{mF^g}& FF^g
\ar@/^/@<+1.5ex>[ll]^{FF^g e'} \ar[r]^-{Fi_F}& FF \ar[r]^-{m}& F
\ar@/^/@<+0.5ex>[ll]^{Fe'}}$$ is a split coequaliser diagram. This means
in particular that for any $a \in \A$,
$$(i_F)_{!}(\phi_{F^g}(a))=(i_F)_{!}(F^g(a), m'_a)=(F(a),
m_a)=\phi_F (a).$$
Thus the left triangle in the diagram
is also commutative.

 Consider the related comonad morphisms

\begin{itemize}
    \item $S_{\phi_{F^g}}: \phi_F U_F \to (i_F)_! (i_F)^*$
    corresponding to the left triangle in (\ref{D.3.3-N}),

\item $S_{\overline{i_F}}:(i_F)_! (i_F)^* \to \widehat{G}$
      corresponding to the right triangle in (\ref{D.3.3-N}),


\item and $S_{\overline{i_F}\cdot \phi_{F^g}}=t:\phi_F U_F \to
    \widehat{G}$ corresponding to the outer diagram in (\ref{D.3.3-N})
\end{itemize}

Then it follows from Proposition \ref{P.3.2-N} that
$t=S_{\overline{i_F}} \cdot S_{\phi_{F^g}}.$
\end{thm}

\begin{lemma}\label{L.3.3-N}
For any $(a,h_a) \in \A_F,$ the $(a,h_a)$-component of
the natural transformation $S_{\phi_{F^g}}$ is the morphism
$$q_a:F(a) \to (i_F)_! ((i_F)^*(a,h_a))=(i_F)_!(a,h_a\cdot (i_F)_a).$$
\end{lemma}

\begin{proof}
Consider the natural transformation $\alpha:\phi_{F^g}U_F \to
(i_F)^*$ corresponding to the left triangle in (\ref{D.3.3-N})
which is the composite
$$
\xymatrix{\phi_{F^g}U_F
\ar[rr]^-{\overline{\eta}\phi_{F^g}U_F}&&(i_F)^*
(i_F)_!\phi_{F^g}U_F=(i_F)^*\phi_{F}U_F
\ar[rr]^-{(i_F)^*\varepsilon_F} && (i_F)^*\,,}$$ where
$\overline{\eta}:1 \to (i_F)^*(i_F)_!$ is the unit of the adjunction
$(i_F)_! \dashv (i_F)^*.$ A simple calculation shows that, for any
$(a,h_a) \in \A_F$,  $\alpha_{(a,h_a)}$ is the composite
$$
\xymatrix{ F^g(a) \ar[r]^-{(i_F)_a} & F(a) \ar[r]^-{h_a}& a.}
$$
Thus, the $(a,h_a)$-component of $S_{\phi_{F^g}}$ is the morphism
$$(i_F)_!(h_a \cdot (i_F)_a): (i_F)_!(\phi_{F^g}U_F(a,h_a)) \to
(i_F)_!((i_F)^*(a,h_a)).$$ Since
$\phi_{F^g}U_F(a,h_a)=\phi_{F^g}(a)=(F^g(a), m'_a)$ and
$(i_F)^*(a,h_a)=(a, h_a \cdot (i_F)_a)$, it follows from the
definition of  $(i_F)_!$ that the diagram
$$
\xymatrix{FF^g F^g(a) \ar[dd]_{FF^g(h_a \cdot (i_F)_a)}
\ar@/^/@<+2.5ex>[rrr]^{F(m'_a)} \ar[rr]_-{F((i_F)_{ F^g(a)})} &&
FFF^g(a) \ar[ddl]_{FF(h_a \cdot (i_F)_a)}\ar[r]_-{m_{F^g(a)}}&
FF^g(a) \ar[dd]_{F(h_a \cdot (i_F)_a)}
\ar[r]^-{q_{F^g(a)}}& (i_F)_!(\phi_{F^g}U_F(a,h_a))\ar[dd]^{(S_{\phi_{F^g}})_{(a,h_a)}}\\\\
FF^g(a) \ar[r]^{F((i_F)_a)}& FF(a) \ar@{->}@<0.5ex>[rr]^-{F(h_a)}
\ar@ {->}@<-0.5ex> [rr]_{m_a}&&F(a) \ar[r]^-{q_a} & (i_F)_!
((i_F)^*(a, h_a))\,,}$$ whose rows are coequaliser diagrams, is
commutative. Note now that the diagram
$$\xymatrix{FF^g F^g  \ar@/^/@<+2.5ex>[rr]^{Fm'} \ar[r]^-{Fi_F F^g} & FFF^g \ar[r]^-{mF^g}& FF^g
\ar@/^/@<+1.5ex>[ll]^{FF^g e'} \ar[r]^-{Fi_F}& FF \ar[r]^-{m}& F
\ar@/^/@<+0.5ex>[ll]^{Fe'}}$$ is a split coequaliser diagram. It
follows that the diagram
$$
\xymatrix{FF^g(a) \ar[rr]^-{m_a \cdot F((i_F)_a)}
\ar[dd]_{F((i_F)_a)} && F(a) \ar[dd]^{(S_{\phi_{F^g}})_{(a,h_a)}}\\\\
F(a) \ar[rr]_{q_a\quad}&&  (i_F)_! ((i_F)^*(a, h_a))}$$
is commutative.
Now, since $q_a \cdot F(h_a) \cdot F((i_F)_a)=q_a \cdot m_a \cdot
F((i_F)_a)$ and since $(S_{\phi_{F^g}})_{(a,h_a)}$ is the unique
morphism making the square commute, we see that
$(S_{\phi_{F^g}})_{(a,h_a)}=q_a.$
\end{proof}

\begin{proposition}\label{P.3.4-N}
Suppose the natural transformation $t:\phi_F U_F \to
\widehat{G}$ to be componentwise a monomorphism. Then $S_{\phi_{F^g}}:
\phi_F U_F \to (i_F)_! (i_F)^*$ is an isomorphism. Thus,
$S_{\overline{i_F}}:(i_F)_! (i_F)^* \to \widehat{G}$ is an
isomorphism if and only if $t$ is so.
\end{proposition}

\begin{proof} First note that, by the previous lemma,
$S_{\phi_{F^g}}$ is a componentwise regular epimorphism. Now,
since any regular epimorphism that is a monomorphism is an
isomorphism and since $t=S_{\overline{i_F}} \cdot S_{\phi_{F^g}}$
(see \ref{nat-trans-S}),
the result follows.
\end{proof}

\begin{thm}\label{Galois-entw}{\bf Galois entwinings.} \em
Write $\widetilde{G}$ for the comonad on the category $\A_F$
generated by the adjunction $(i_F)_! \dashv (i_F)^*$ and let $t_g :
\widetilde{G} \to \widehat{G}$ be the related comonad morphism
  (see \cite[Theorem 4.1]{M}). This leads to a commutative
diagram with the canonical comparison functor $K_{\widetilde{G}}$
(e.g. \cite[Lemma 4.3]{M})
$$
\xymatrix{\A_{F^g} \ar[rd]_{\overline{i_F}}
\ar[r]^{K_{\widetilde{G}}}&
(\A_F)^{\widetilde{G}} \ar[d]^{(\A_F)_{t_g}}\\
 & (\A_F)^{\widehat{G}}. }$$

By Definition \ref{def-galois}, the functor $(i_F)_!$ is
$\widehat{G}$-Galois provided $t_g : \widetilde{G} \to \widehat{G}$
is an isomorphism. If this is the case we call $(F,G,\lambda,g)$ a
{\em Galois entwining} and $g:I\to G$ a \emph{Galois (grouplike)
morphism} and we have:
\end{thm}

\begin{theorem}\label{Gal-entw} Let $\lambda : FG \to GF$ be an
entwining from a monad $\textbf{F}$ to a comonad $\textbf{G}$ on a
category $\A$. Suppose that $g: I \to G$ is a grouplike morphism
such that the corresponding functor $(i_F)^*:\A_{F} \to \A_{F^g}$
admits a left adjoint functor $(i_F)_!:\A_{F^g} \to \A_F$. Then the
comparison functor $\overline{i_F}:\A_{F^g}\to (\A_F)^{\widehat{G}}$
is an equivalence of categories if and only if $(\textbf{F},
\textbf{G}, \lambda, g)$ is a Galois entwining and the functor
$(i_F)_!$ is comonadic.
\end{theorem}

 In the situation of the preceding
theorem, if $g$ is such that the corresponding comparison functor $K_g:
\A \to (\A_F)^{\widehat{G}}$ is full and faithful, it follows from
Proposition \ref{P.3.6.x} that the functor $\overline{i_F}$ reduces to the
functor $K_g$.

\section{Bimonads}

\begin{thm}\label{bimonad}{\bf Properties of bimondas.} \em
Recall from \cite[Definition 4.1]{MW} that
a  bimonad $\bH$ on a category $\A$ is an endofunctor $H : \A \to \A$
 which has a monad structure
 $\uH=(H, m, e)$ and a comonad structure $\oH=(H, \delta, \varepsilon)$
with an entwining $\lambda:HH\to HH$ inducing commutativity of the
diagrams
\begin{equation}\label{D.1.18b}
 \xymatrix{ HH  
 \ar@{->}@<0.5ex>[r]^-{\ve H} \ar@ {->}@<-0.5ex>[r]_-{H\ve}
\ar[d]_m
& H \ar[d]^\varepsilon \\
    H \ar[r]^\varepsilon & 1 , } \qquad
\xymatrix{
1 \ar[r]^-{e} \ar[d]_-{e} & H \ar[d]^-{\delta}\\
H 
\ar@{->}@<0.5ex>[r]^-{eH} \ar@ {->}@<-0.5ex>[r]_-{He}&
 HH,}
 \qquad
\xymatrix{ 1\ar[r]^e \ar[dr]_= & H \ar[d]^\vareps\\
         & 1 ,}
\end{equation}

\begin{equation}\label{D.1.18}
\xymatrix{
HH \ar[r]^-{m} \ar[d]_-{H \delta}& H \ar[r]^-{\delta}& HH \\
HHH \ar[rr]_-{\lambda H}&& HHH  \ar[u]_-{Hm}.}
\end{equation}

Joining $H$ from the left to the central diagram in (\ref{D.1.18b})
and attaching the resulting square on the left hand side of (\ref{D.1.18}),
one derives the relation
\begin{equation}\label{lambda-d}
 \lambda \cdot He = \delta.
\end{equation}

For the bimonad $\bH$ we obtain the comparison functor
$$K_H : \A \to \A_H^H,\,\, a \longrightarrow(H(a), m_a, \delta_a) ,$$
where $\A_H^H = \A_{\underline{H}}^{\overline{H}}(\lambda) $, with
commutative diagrams
\begin{equation}\label{Comp-ff}
\xymatrix{ \A \ar[r]^-{K_H} \ar[rd]_-{\phi_{\underline{H}}} &
 (\A_{\underline{H}})^{\widehat{\overline{H}}}
\ar[d]^-{U^{\widehat{\overline{H}}}}  \simeq \A_H^H \\
& \A_{\underline{H}}} ,\qquad
 \xymatrix{
\A \ar[r]^-{K_H} \ar[rd]_-{\phi^{\overline{H}}} &
(\A^{\overline{H}})_{\widehat{\underline{H}}}
\ar[d]^-{U_{\widehat{\underline{H}}}}\simeq \A_H^H\\
& \A^{\overline{H}}.}\end{equation} As noticed in \cite[5.13]{W},
the comparison functor $K_H$ is full and faithful by the isomorphism
 $$\Mor_H^H(H(a),H(b))\to \Mor_\A(a,b),\quad f\mapsto \ve_b\circ f\circ e_a.$$
\end{thm}

We now reconsider bimonads and Hopf monads in view of the notions
introduced in the preceding sections.

It is clear from (\ref{D.1.18b}) that the unit $e:I\to H$ is a
grouplike morphisms (as defined in \ref{grouplike}).  Write $\gamma$
for the composite $Hm \cdot \delta H$. Then since $\gamma \cdot He=
\delta$ (see \cite[(5.2)]{MW}), it is easy to see that the functor
$K_H$ is just the functor $K_e$ corresponding to the grouplike
morphisms $e: 1 \to H$. Then, since the functor $K_H$ is full and
faithful, it follows from Proposition \ref{P.3.6.x} that the diagram
$$\xymatrix{ 1 \ar[r]^e &H\ar@{->}@<0.5ex>[rr]^-{eH } \ar@ {->}@<-0.5ex>
[rr]_-{\delta=\lambda \cdot He}&& H H }$$
is an equaliser diagram.
Therefore the functor $F^e$ from \ref{equal-func}, that is
$\uH^\oH$, is just the identity on $\A$. Thus $(i_\uH)^*$ turns out
to be the forgetful functor $U_\uH:\A_\uH\to \A$ and its left
adjoint $(i_\uH)_!$ is the free functor $\phi_\uH:\A\to \A_\uH$.
Now, since the unit of the adjunction $\phi_H \dashv U_H$ is a split
monomorphism, the functor $\phi_\uH:\A\to \A_\uH$ is always
comonadic, provided the category $\A$ is Cauchy complete (see
Corollary 3.19 in \cite{Me}), it follows from \ref{Gal-entw}:

\begin{thm}\label{H-galois}{\bf $\phi_H$ as $\woH$-Galois functor.}
For a 
bimonad $\textbf{H}$ on a Cauchy complete category
$\A$, the following are equivalent:
\begin{blist}
\item $\phi_H$ is a $\woH$-Galois functor;
\item the unit $e:I\to H$ is a Galois grouplike morphism;
\item  the functor $K_H : \A \to \A^H_H$ is an
equivalence of categories.
\end{blist}
\end{thm}

\begin{proposition}\label{P.Bim.1}
Assume that $\A$ admits equalisers and that $H$ has a right adjoint.
Then the following are equivalent:
\begin{blist}
\item  the functor $K_H : \A \to \A^H_H$ is an
equivalence of categories;
\item   $(H, m)$ is $\overline{H}$-Galois;
\item   $\textbf{H}$ has an antipode.
\end{blist}
\end{proposition}
\begin{proof} Clearly (a) implies (b), while the equivalence
of (a) and (c) is proved in \cite[5.6]{MW}. So suppose that $(H,m)$
is $\oH$-Galois. Then it follows from Proposition \ref{P.2.4} that
$\phi_{\underline{H}}$ is $\woH$-Galois, i.e. the comonad morphism
$t_{\phi_{\underline{H}}}: \phi_{\underline{H}}U_{\underline{H}} \to
\widehat{\overline{H}}$ is an isomorphism. Now, since   the
category $\A$ admits equalisers, it is Cauchy complete, and as it
was noted above,  the functor $\phi_{\underline{H}}$ is always
comonadic, it follows from Proposition \ref{P.1.7} that $K_H$ is an
equivalence of categories. This completes the proof.
\end{proof}

Dually, one has

\begin{proposition}\label{P.Bim.2}
Assume that $\A$ admits coequalisers and that $H$ has a left
adjoint. Then the following are equivalent:
\begin{blist}
\item  the functor $K_H : \A \to \A^H_H$ is an
equivalence of categories;
\item   $(H, \delta)$ is $\underline{H}$-Galois;
\item   $\textbf{H}$ has an antipode.
\end{blist}
\end{proposition}

Combining the propositions \ref{P.2.4}, \ref{P.2.5}, \ref{P.Bim.1} and \ref{P.Bim.2}, we get
\begin{theorem}\label{T.Bim}
Assume that
\begin{rlist}
\item $\A$ has small limits or
colimits and $H$ preserves them, or
\item $\A$ admits
equalisers and $H$ has a right adjoint, or
\item $\A$
admits coequalisers and $H$ has a left adjoint.
\end{rlist}
Then the functor
$K_H : \A \to \A^H_H$ is an equivalence of categories if and only if
$\textbf{H}$ has an antipode.
\end{theorem}

\section{Bimonads in the sense of A. Brugui\`{e}res and A.
Virelizier}

Let $(\bV, \otimes, \II)$ be a strict monoidal category.

\begin{thm}\label{BV-mon}{\bf BV-bimonads.} \em
Let $\textbf{T}=(T,m, e)$ be a monad on $\V$, such that
the functor $T$ and the natural transformations $m$ and $e$ are comonoidal,
that is, there are natural transformations
\begin{center}
$\chi_{X,Y} : T(X\ot Y)\to T(X)\ot T(Y)$ for $X,Y\in \V$
\end{center}
and a morphism $\theta_\II : T(\II)\to \II$ satisfying certain
compatibility axioms. Such monads are named {\em bimonads} by
Brugui\`{e}res and   Virelizier in \cite[Section 2.3]{BV} and we
call them {\em BV-bimonads} to avoid confusion with other notions of
"bimonads".
 It follows from the definition
that the triple
$$(T(\II),\; \chi_{\II,\,\II}: T(\II) \to T(\II) \otimes T(\II),\;
\theta_\II: T(\II)\to \II)$$
 is a coalgebra in $\V$ (see \cite[p. 704]{BV}),
and thus one has a comonad $\textbf{G}$ on $\V$ with functor  $G=
- \otimes T(\II)$. Then the compatibility axioms ensure that
the natural transformation
$$\lambda:=(T(-) \otimes m_\II)\cdot \chi_{-,\, T(\II)}: TG \to GT,$$ is
a mixed distributive law (entwining) from the monad $\textbf{T}$ to the comonad
$\textbf{G}$.
\end{thm}

\begin{thm}\label{BV-mod}{\bf BV-Hopf modules.} \em
The {\em entwining modules} are objects $M \in \V$ with a
$T$-module structure $h:T(M)\to M$ and a comodule structure
$\rho:M\to M\ot T(\II)$ inducing commutativity of the diagram
$$\xymatrix{
 T(M) \ar[rr]^h \ar[d]_{T(\rho)} & &
        M \ar[rr]^{\rho} & &  M\ot T(\II) \\
T(M\ot T(\II)) \ar[rr]^{\chi_{M,T(\II)}}& & T(M)\otimes TT(\II)
  \ar[rr]^{T(M)m_\II}& &
                  T(M)\otimes T(\II) \ar[u]_{h \otimes T(\II)}.}$$
These are named {\em right Hopf $T$-modules} in \cite[Section
4.2]{BV} and we call them {\em right BV-Hopf modules}. Their
category  is just $\V^G_T$.

{}From the ingredients of the definition one obtains the commutative diagram
$$\xymatrix{
 TT \ar[rr]^m \ar[d]_{T\chi_{-,\,\II}} & &
        T \ar[rr]^{\chi_{-,\II}} & & T(-)\otimes T(\II) \\
T(T(-)\otimes T(\II)) \ar[rr]^{\chi_{T(-),T(\II)}}& & TT(-)\otimes TT(\II)
  \ar[rr]^{TT(-) \otimes m_\II}& &
                 TT(-)\otimes T(\II) \ar[u]_{m \otimes T(\II)}}$$
which shows that for any $X\in \V$, $T(X)$ is a right BV-Hopf module leading to the
commutative diagram
$$\xymatrix{\V \ar[r]^-{K} \ar[rd]_{\phi_T}&
(\V_T)^{\wG}=\V^G_T \ar[d]^{ U^{\wG} }\\
& \V_T\, ,}$$
 with a comparison functor $K(X)=(T(X), m_X, \chi_{X,\, I})$.

For the corresponding comonad morphism
$t_K :\phi_TU_T \to \widehat{G}$, it is easy to see that for any $(X, h_X)
\in V_T$, the $(X, h_X)$-component of $t_K$ is the
composite
$$\xymatrix{T(X) \ar[r]^-{\chi_{X,\,\II}}& T(X) \otimes T(\II)
\ar[rr]^-{h_X \otimes T(\II)}&& X \otimes T(\II).}$$
Since $t_K$ is a comonad morphism, we have the commutative diagram
$$
\xymatrix{ \phi_T U_T \ar[r]^-{t_K} \ar[rd]_{\varepsilon_T}
&\widehat{G} \ar[d]^{\varepsilon_{\widehat{G}}}\\
& 1,}$$
and since, for any $(X,h_X) \in \bV$,
$(\varepsilon_T)_{(X, h_X)}=h_X$ and
$(\varepsilon_{\widehat{G}})_{(X,h_X)}=X \otimes \theta_{\II}$, we
have
\end{thm}

\begin{lemma}\label{L.5.1} For any $(X, h_X)\in \V_T$,
$$(X \otimes \theta_\II) \cdot (t_K)_{(X,
\,h_X)}=h_X.$$
\end{lemma}


\begin{thm}\label{rem.5.1}{\bf Remark.} \em Note that there are also functors
$$\begin{array}{rl}
\V_T\to \V^G_T, & (M,h) \mapsto (M\ot T(\II),
\tilde h:=(h\ot m_\II)\circ\chi_{M,T(\II)}
, M\ot \chi_{\II,\II}) \\[+1mm]
\V^G\to \V^G_T, & (N,\rho) \mapsto
(T(N),m_N,\widehat \rho:=(T(N)\ot m_\II)\circ\chi_{N,T\II}\circ T(\rho)).
\end{array}$$

The second functor corresponds to \cite[Lemma 4.3]{BV}.
\end{thm}

\bigskip
\begin{thm}\label{g-like} {\bf Grouplike morphism.} \em
  Since $T$ is a comonoidal monad on $\V$, the following two
diagrams
$$\xymatrix{\II  \ar[r]^-{e_{\II}}  \ar@{=}[rd] & T(\II)
\ar[d]^-{\theta_{\II}} &\text{and}& \II  \ar[r]^-{e_{\II}}
\ar[rd]_{e_{\II}\ot e_{\II}} &T(\II) \ar[d]^{\chi_{\II,X}} \\
& \II & & &T(\II)\ot T(\II)}$$
both are commutative, implying that
the natural transformation $$g:=- \ot e_{\II}: 1 \to - \ot T(\II)$$
is a grouplike morphism. Note that $gT : T \to T \ot T(\II)$ is the
natural transformation given by $T(X) \ot e_{\II}: T(X) \to T(X)\ot
T(\II)$, while $\lambda \cdot Tg : T \to T \ot T(\II)$ is given by
the composite
$$\xymatrix{ T(X) \ar[rr]^-{T(X \ot e_{\II})} && T(X \ot T(\II)) \ar[rr]^-{\chi_{X,
T(\II)}}&& T(X) \ot TT(\II)  \ar[rr]^-{T(X) \ot m_{\II}}&& T(X) \ot
T(\II)}.$$ Since $T$ is a comonoidal monad, the diagram
$$\xymatrix{T(X) \ar[rr]^-{T(X \ot e_{\II})}\ar[d]_{\chi_{X, \II}}&& T(X \ot T(\II)) \ar[d]^-{\chi_{X,
T(\II)}}\\
T(X) \ot T(\II) \ar[rr]_-{T(X) \ot T(e_{\II})}& & T(X) \ot
TT(\II)}$$is commutative. But since $m_{\II}\cdot e_{\II}=1$, we see
that $\lambda \cdot Tg$ is just the natural transformation $\chi_{-,
\II}$. Thus, for any $X \in \V$, $T^g(X)$ is the equaliser
$$\xymatrix{ T^g(X) \ar[r]&  T(X)
\ar@{->}@<0.5ex>[rr]^-{T(X) \ot e_{\II}} \ar@{->}@<-0.5ex>
[rr]_-{\chi_{X, \II}}&& T(X)\otimes T(\II).}$$

  Note that the functor $K : \V \to \V^G_T$ is just the functor
$K_g: \V \to \V^G_T=(\V_T)^{\widehat{G}}.$
\end{thm}

\begin{thm}\label{BVH}{\bf BV-Hopf monads.} \em From now on, we suppose that $\textbf{T}$ is a right Hopf monad in
the sense of \cite[Section 3.6]{BV} on a right autonomous category $\V$,
we call it a {\em BV-Hopf monad}.

Consider the natural transformation $\Gamma: G \to TT$ defined in
\cite[Section 4.5]{BV}. We shall need the following simple properties of this
functor (see \cite[Lemma 4.9]{BV}):
\begin{equation}\label{E.16}
     m \cdot \Gamma= e \otimes \theta_\II.
\end{equation}
\begin{equation}\label{E.17}
Tm \cdot \Gamma T \cdot \chi_{-,\, \II}=T e.
\end{equation}

Using these, one can calculate (see \cite{BV}) that for any
$(X, h_X, \vth_X) \in \V_T^G$,
\begin{equation}\label{E.18}
h_X \cdot T(h_X)\cdot \Gamma_X \cdot \vth_X =1_X
\end{equation}
\end{thm}

\begin{lemma}\label{L.5.7}For any
$(X, h_X, \vth_X) \in \V_T^G$, the morphism
$$(t_K)_{(X, \,h_X)}=(h_X \otimes T(\II)) \cdot \chi_{X,\,\II}:T(X) \to X
\otimes T(\II)$$
is a split monomorphism.
\end{lemma}
\begin{proof}For any
$(X, h_X, \vth_X) \in  \V_T^G$, consider the
composite
$$q_{(X, \, h_X)}=T(h_X) \cdot \Gamma_X : X \otimes T(\II)\to T(X).$$
We claim that $q_{(X, \, h_X)} \cdot (t_K)_{(X,
\,h_X)}=1.$ Indeed, consider the diagram

$$\xymatrix{T(X) \ar[rr]^-{\chi_{X, \,\II}}\ar[rrdddd]_{T(e_X)}&& T(X)
\otimes T(\II) \ar@{}[rrdd]^{(1)} \ar[rr]^{h_X \otimes T(\II)}
\ar[dd]^{\Gamma_{T(X)}}&& X\otimes T(\II) \ar[dd]^{\Gamma_{X}}\\\\
&& TTT (X) \ar[rr]_{TT(h_X)} \ar[dd]^{T(m_X)} && TT(X)
\ar[dd]^{T(h_X)}\\\\
&& TT(X) \ar@{}[rruu]_{(2)}\ar[rr]_{T(h_X)} && T(X)\,.}$$ In this
diagram

\begin{itemize}
    \item square (1) commutes because $\Gamma$ is a functor,
    \item square (2) commutes because $(X, h_X)$ is a
    $\textbf{T}$-algebra, and
    \item the triangle commutes because of (\ref{E.17}).
\end{itemize} It follows that
$$q_{(X, \, h_X)} \cdot (t_K)_{(X, \, \,h_X)}=T(h_X) \cdot T(e_X)=T(h_X \cdot
e_X)=1_X.$$ Thus
\begin{equation}\label{E.19}
q_{(X, \, h_X)} \cdot (t_K)_{(X, \,\,h_X)}=1_X .
\end{equation}
\end{proof}

Since, by Lemma \ref{L.5.7}, $t_K$ is a componentwise (split)
monomorphism,
 Proposition \ref{P.3.4-N} yields the

\begin{thm}\label{Cor.5.1}{\bf Corollary.}
$t_K :\phi_TU_T \to \widehat{G}$ is an isomorphism if and only if
$e_\II : \II \to T(\II)$ is a Galois grouplike morphism.
\end{thm}

\begin{proposition}\label{P.3.3.x} For any $(X, h_X, \vth_X) \in
\V_T^G$, the diagram
\begin{equation}\label{E.20}
\xymatrix{ X \ar[rr]^-{q_{(X, h_X)}\cdot \vth_X}&&T(X)
\ar@{->}@<0.5ex>[rr]^-{T(\vth_X)} \ar@ {->}@<-0.5ex>
[rr]_-{T(X\otimes e_\II )}&& T(X\otimes T(\II)),}\end{equation} is a
split equaliser diagram.
\end{proposition}
\begin{proof}   Note first that, by \cite[Lemma 4.11]{BV}, the
composite $q_{(X, h_X)}\cdot \vth_X$ equalises the pair
$(T(\vth_X), T(X\otimes e_\II ))$. Next  the following diagram
is serially commutative (see \cite{G})
\begin{equation}\label{E.21}
\xymatrix{& &T(X) \ar[dd]_{s_1}\ar@{->}@<0.5ex>[rr]^-{T(\vth_X)}
\ar@ {->}@<-0.5ex> [rr]_-{T(X \otimes e_\II)}&& T(X \otimes T(\II))
\ar[dd]^{s_2}\\\\
X \ar[rr]_-{ \vth_X}&& X \otimes T(\II)
\ar@{->}@<0.5ex>[rr]^-{\vth_X \otimes T(\II)} \ar@ {->}@<-0.5ex>
[rr]_-{ X \otimes \chi_{\II,\,\II}}&& X \otimes T(\II)\otimes
T(\II)\,}
\end{equation}
where $s_1=(t_K)_{(X,\,h_X)}$ and
$s_2=(t_K)_{(X \otimes T(\II),\, h_{X \otimes T(\II)})}.$ Note that
the bottom row of this diagram is split by the morphisms  $X \otimes
\theta_\II$ and $X \otimes T(\II) \otimes \theta_\II$. Recall that
this means
\begin{equation}\label{E.22} (X\otimes \theta_\II) \cdot \vth_X=1,
\end{equation}
\begin{equation}\label{E.23}
(X \otimes T(\II) \otimes \theta_\II) \cdot (X \otimes \chi_{\II,\,\II})=1,
\mbox{ and }
\end{equation}
\begin{equation}\label{E.24}
(X \otimes T(\II) \otimes \theta_\II) \cdot (\vth_X
\otimes T(\II))=\vth_X \cdot (X \otimes \theta_\II).
\end{equation}
By \ref{E.18}, we now have
$$ h_X \cdot q_{(X, \, h_X)} \cdot \vth_X=h_X \cdot T(h_X) \cdot \Gamma_X
\cdot \vth_X=1_X.$$
Furthermore,  since $s_2 \cdot T(X \otimes e_\II)
 = (X \otimes \chi_{\II,\,\II}) \cdot s_1$,
$$\begin{array}{l}
q_{(X, \, h_X)}\cdot (X \otimes T(\II) \otimes \theta_\II) \cdot s_2
\cdot T(X \otimes e_\II)   \\[+1mm]
\hspace{2cm}= \; q_{(X, \, h_X)}\cdot (X\otimes T(\II) \otimes \theta_\II) \cdot (X \otimes \chi_{\II,\,\II}) \cdot s_1 \\[+1mm]
\hspace{2cm} =_{(\ref{E.23})} q_{(X, \, h_X)}\cdot s_1=q_{(X, \,
h_X)}\cdot (t_K)_{(X,\, h_X)}  =_{(\ref{E.19})}  1_X ,
 \end{array}  $$
and since $s_2 \cdot T(\vth_X)=(\vth_X \otimes T(\II)) \cdot s_1$,
$$\begin{array}{l}
q_{(X, \, h_X)}\cdot (X \otimes T(\II) \otimes \theta_\II) \cdot s_2
\cdot T(\vth_X) \\
\hspace{2cm} = \;
q_{(X,\,h_X)}\cdot (X\ot T(\II)\ot \theta_\II)\cdot (\vth_X \otimes
T(\II)) \cdot s_1 \\
\hspace{2cm} = _{(\ref{E.24})}
   q_{(X,\, h_X)}\cdot\vth_X\cdot (X\ot \theta_\II)\cdot s_1 \\
\hspace{2cm} = \; q_{(X,\,h_X)}\cdot\vth_X \cdot(X \ot \theta_\II) \cdot
     (t_K)_{(X,\,h_X}) \\
\hspace{2cm}  =_{L.\,\ref{L.5.1}} q_{(X, \, h_X)}\cdot \vth_X \cdot  h_X.
\end{array} $$
We have proved that
\begin{equation}
 h_X \cdot q_{(X, \, h_X)} \cdot \vth_X=1_X,
\end{equation}
\begin{equation}
q_{(X, \, h_X)}\cdot (X \otimes T(\II) \otimes \theta_\II) \cdot s_2
\cdot T(X \otimes e_\II)=1_X, \mbox{ and }
\end{equation}
\begin{equation}
q_{(X, \, h_X)}\cdot (X \otimes T(\II) \otimes \theta_\II) \cdot s_2
\cdot T(\vth_X)=q_{(X, \, h_X)}\cdot \vth_X  \cdot h_X,
\end{equation}
which just means that (\ref{E.20}) is a split equaliser: a splitting
is given by $h_X$ and by $q_{(X, \, h_X)}\cdot (X\otimes T(\II)
\otimes \theta_\II) \cdot s_2$.
\end{proof}

\begin{proposition} The functor $K:\V \to \V_T^G$
has a fully faithful right adjoint if and only if for any $(X, h_X, \vth_X)
\in \V_T^G$, the pair of morphisms
\begin{equation}\label{E.28}
\xymatrix{X \ar@{->}@<0.5ex>[r]^-{\vth_X} \ar@ {->}@<-0.5ex> [r]_-{X
\otimes e_\II}& X\otimes T(\II)}\end{equation} has an equaliser and
this equaliser is preserved by $T$.
\end{proposition}
\begin{proof} By \ref{right-adj},
$K$ has a right adjoint if and only if (\ref{E.28}) has an equaliser
for all $(X, h_X,\vth_X) \in \V_T^G$. We
write $(\overline{X}, i_X :\overline{X}\to X)$ for this equaliser.
Thus $R(X, h_X, \vth_X)=(\overline{X}, i_X)$. Since the diagram
(\ref{E.21}) is commutative and since $T(i_X)$ equalises $T(\vth_X)$
and $T(X \otimes e_\II)$, there exists a unique morphism $k_X=k_{(X,
\,h_X, \,\vth_X)}:T(\overline{X})\to X $ making the diagram
\begin{equation}\label{E.29}
 \xymatrix{T(\overline{X})\ar[rr]^-{T(i_X)}
\ar[dd]_{k_X}&& T(X) \ar[dd]_{s_1}\ar@{->}@<0.5ex>[rr]^-{T(\vth_X)}
\ar@ {->}@<-0.5ex> [rr]_-{T(X \otimes e_\II)}&& T(X \otimes T(\II))
\ar[dd]^{s_2}\\\\
X \ar[rr]_-{\vth_X}&& X \otimes T(\II) \ar@{->}@<0.5ex>[rr]^-{\vth_X
\otimes T(\II)} \ar@ {->}@<-0.5ex> [rr]_-{X \otimes \delta_\II}&& X
\otimes T(\II)\otimes T(\II)\,}
\end{equation}
commute. Since $q_{(X,\,h_X)}\cdot \vth_X\cdot k_X =q_{(X,
\,h_X)}\cdot s_1 \cdot T(i_X)=T(i_X)$ and since $(X,q_{(X,
h_X)}\cdot \vth_X)$ is an equaliser of the pair $(T(\vth_X), T(X
\otimes e_\II))$ by Proposition \ref{P.3.3.x}, it follows from the
universal property of equalisers that $k_X$ is an isomorphism if and
only if the top row of diagram (\ref{E.29}) is an equaliser diagram,
i.e. if $T$ preserves the equaliser of (\ref{E.28}). Since according
to \cite{G}, $k_X=k_{(X, \,h_X, \,\vth_X)}$ is the $(X, h_X,
\vth_X)$-component of the counit $\overline{\varepsilon}$ of the
adjunction $K \dashv R$ and since $R$ is full and faithful if and
only if $\overline{\varepsilon}$ is an isomorphism, it follows that
$R$ is a fully faithful functor if and only if for any $(X, h_X,
\vth_X) \in \V_T^G$, the pair of morphisms $(\vth_X,X \otimes
e_\II)$ has an equaliser and this equaliser is preserved by $T$.
\end{proof}
\begin{theorem}\label{T.5.5} The functor $K:\V \to \V_T^G
$ is an equivalence of categories if and only if the functor $T$ is
conservative (=isomorphism reflecting) and for any $(X, h_X, \vth_X)
\in \V_T^G$, the pair of morphisms $(\vth_X,X
\otimes e_\II)$ has an equaliser and this equaliser is preserved by
$T$.
\end{theorem}
\begin{proof} According to the previous proposition it is enough to
show that the fully faithful  functor $R$ is an equivalence of
categories if and only if  $T$ is conservative. But since any fully faithful
functor with a left adjoint is an equivalence of categories if and only if the
left adjoint is conservative, it is sufficient to prove that $T$ is
conservative if and only if the functor $K$ is, which is indeed the case since
$T=U_T\phi_T=U_T U^{\widehat{G}}K$ and the functors $U_T$ and $
U^{\widehat{G}}$ are both conservative.
\end{proof}

\smallskip

Recall (e.g. \cite{Me}) that a monad $\textbf{T}$  on an arbitrary
category $\A$ is  of \emph{effective descent type} if the functor
$\phi_T : \A \to \A_T$ is comonadic.

\begin{theorem}\label{T.5.6} For any right BV-Hopf monad  $\textbf{T}$ on a right
autonomous Cauchy complete monoidal category $\V$, the functor
$K:\V\to \V_T^G$ is an equivalence
if and only if $\textbf{T}$ is of effective descent type.
\end{theorem}
\begin{proof} If $K$ is an equivalence of categories, then the
functor $\phi_T$ is comonadic by Proposition \ref{P.1.7}.

Conversely, suppose that $\textbf{T}$ is of effective descent type.
Since $\V$ is Cauchy complete, it follows from \cite[Proposition
3.11]{Me} that $\textbf{T}$ is of effective descent type if and only if  $T$ is
conservative and $\V$ has equalisers of $T$-split pairs and these
equalisers are preserved by $T$. Now, if $(X, h_X, \vth_X) \in
\V_T^G$, then the pair of morphisms
$(T(\vth_X),T(X \otimes e_\II))$ is split by Proposition
\ref{P.3.3.x}
 and thus there exists an equaliser $(\overline{X}, i_X)$ of the pair
$(\vth_X,X \otimes e_\II)$ and this equaliser is preserved by $T$.
The preceding theorem completes the proof.
\end{proof}

  In the light of Proposition \ref{P.1.7} and Corollary
\ref{Cor.5.1}, we have:

\begin{corollary}If a right BV-Hopf monad  $\textbf{T}$ on a right
autonomous Cauchy complete monoidal category $\V$ is of effective
descent type, then the natural transformation $t_K :\phi_TU_T \to
\widehat{G}$ is an isomorphism of comonads. Moreover, $e_\II: \II
\to T(\II)$ is a Galois grouplike morphism.
\end{corollary}

Since any monad on a Cauchy complete category  whose unit is a
split monomorphism is of effective descent type (see  \cite{Me}),
it follows from Theorem \ref{T.5.6} that

\begin{corollary}\label{C.5.7} For any right BV-Hopf monad  $\textbf{T}=(T,m,e)$ on a right
autonomous Cauchy complete monoidal category $\V$ with $e: 1\to T$ a
split monomorphism, the functor $K:\V \to \V_T^G$ is an equivalence
\end{corollary}

 Combining Proposition \ref{P.3.6.x}, Theorem \ref{T.5.6} and
Corollary \ref{C.5.7}, we get

\begin{proposition}\label{P.5.8}
Let $\textbf{T}=(T,m,e)$ be a right BV-Hopf monad on a right
autonomous Cauchy complete monoidal category $\V$. If
$\textbf{T}$ is of effective descent type, then the monad
$\textbf{T}^g$ is (isomorphic to) the identity monad. In particular,
this is the case provided that the unit $e: 1\to T$ is a split
monomorphism.
\end{proposition}

\begin{thm}\label{braid}{\bf Bimonads in braided categories.}
\em
 As before, let $(\bV, \otimes, \II)$ be a strict monoidal category and
$\textbf{T}=(T,m, e)$ a comonoidal monad on $\bV$, and consider the
corresponding mixed distributive law (entwining)
$$\lambda:=(T(-) \otimes m_\II)\cdot \chi_{-,\, T(\II)}: TG \to
GT,$$ from the monad $\textbf{T}$ to the comonad $\textbf{G}=-
\otimes T(\II)$. It is pointed out in \cite{BV} that, when $\bV$ is
a {\em braided monoidal category} with braiding $\tau_{X, Y}: X \otimes Y
\to Y \otimes X$, then for any bialgebra $\textbf{A}=(A,
e,m,\varepsilon, \delta)$ in $\bV$, the monad $A \otimes - $ is a
comonoidal monad, where the natural transformation $\chi_{X,Y} : A
\otimes X \otimes Y \to A \otimes X \otimes A  \otimes Y$ is the
composite
$$
\xymatrix{A \otimes X \otimes Y \ar[rr]^-{\delta \otimes X \otimes
Y}&& A \otimes A \otimes X \otimes Y \ar[rr]^-{A \otimes \tau_{A,
X}\otimes Y }&& A \otimes X \otimes A \otimes Y \, .}$$
Then, for any $X \in \bV$, $\lambda_X$ is the composite
$$
\xymatrix{A \otimes X \otimes A \ar[r]^-{\delta \otimes X \otimes
A}& A \otimes A \otimes X \otimes A \ar[rr]^-{A \otimes \tau_{A,
X}\otimes X }&& A \otimes X \otimes A \otimes A \ar[r]^-{A \otimes X
\otimes m}&  A \otimes X \otimes A \,.}$$
Consider now the diagram
$$
\xymatrix{A \otimes X \otimes A \ar[dd]_{\tau_{A, X}\otimes A}
\ar[r]^-{\delta \otimes X \otimes A}& A \otimes A \otimes X
\otimes A \ar@{}[rdd]^{(2)}\ar[dd]_{\tau_{A \otimes A, X}\otimes
A} \ar[rr]^-{A \otimes \tau_{A, X}\otimes X }&& A \otimes X
\otimes A \otimes A \ar@{}[dd]^{(3)}\ar[lldd]^{\tau^{-1}_{A,
X}\otimes A \otimes A}\ar[r]^-{A \otimes X \otimes m}& A \otimes X
\otimes A \ar[ldd]^{\tau^{-1}_{A, X}\otimes
A}\\\\
X \otimes A \otimes A \ar@{}[ruu]^{(1)}\ar[r]_-{X \otimes \delta
\otimes A}& X \otimes A \otimes A \otimes A \ar[rr]_-{X \otimes A
\otimes m}& &X \otimes A \otimes A\, ,}$$
in which the diagrams (1) and
(2) commute by naturality of $\tau$, while diagram (3) commutes by
naturality of composition. Since each component of $\tau$ is an
isomorphism, $\lambda_X$ is an isomorphism if and only if the composite
$(X\otimes A \otimes m)(X \otimes \delta \otimes A)$ is so. Since
$(X\otimes A \otimes m)(X \otimes \delta \otimes A)=X \otimes ((A
\otimes m)(\delta \otimes A))$ and since $(A \otimes m)(\delta
\otimes A)$ is an isomorphism if and only if $\textbf{A}$ has an antipode, it
follows that the composite $(X \otimes A \otimes m)(X \otimes \delta
\otimes A)$ - and hence $\lambda_X$- is an isomorphism for all $X \in
\bV$ if and only if $\textbf{A}$ has an antipode.
\end{thm}


\section{Categories with finite products and Galois objects}

In the category $\Set$ of sets, for any object $G$, the product
$G\times-$ defines an endofunctor. This is always a comonad with the
coproduct given by the diagonal map, and it is a monad provided
$G$ is a semigroup. In this case $G\times-$ is a (mixed) bimonad
and it is a Hopf monad if and only if $G$ is a group. We refer to
\cite[5.19]{W} for more details.

In this final section we study similar operations in more general
categories and this leads eventually to the {\em Galois objects}
in such categories as studied in Chase and Sweedler \cite{CS}.
\smallskip

Let $\A$ be a category with finite products. In particular, $\A$ has
a terminal object, which is the product over the empty set. Then
$(\A,\times, \texttt{1})$ is a symmetric monoidal category, where $a
\times b$ is some chosen product of $a$ and $b$, and $\texttt{1}$ is
a chosen terminal object in $\A$, while the symmetry $\tau_{a, b}:a
\times b \to b \times a$ is the unique morphism for which the
diagram
$$
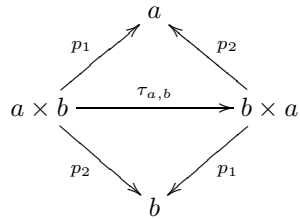
\xymatrix{&a& \\a\times b \ar[ru]^{p_1}\ar[rd]_{p_2}\ar[rr]^{\tau_{a, b}}&& b
\times a \ar[ld]^{p_1}\ar[lu]_{p_2}\\
& b& }$$ commutes. The associativity and unit constraints are
defined via the universal property for products.   Such a
category is called a \emph{cartesian monoidal category}.

Similarly, a \emph{cocartesian monoidal category} is a monoidal
category whose monoidal structure is given by the categorical
coproduct and  whose unit object is the initial object. Any
category with finite coproducts can be considered as a cocartesian
monoidal category.

Given morphisms $f: a \to x$ and $g: a \to y$ in $\A$, we write
$<f,g>:a \to x \times y$ for the unique morphism making the diagram
$$
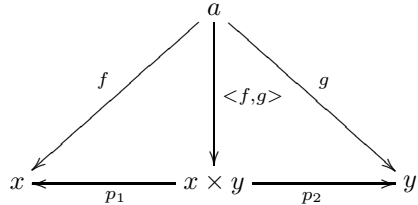
\xymatrix{&& a \ar[lldd]_{f} \ar[rrdd]^{g} \ar[dd]^{<f,g>}&&\\\\
x && x \times y \ar[ll]^{p_1} \ar[rr]_{p_2}&& y}$$ commute. In
particular, $\Delta_a=<1_a, 1_a> : a \to a \times a$ is the diagonal
morphism.

It is well known that every object $c$ of $\A$ has a unique
(cocommutative) comonoid structure in the monoidal category
$(\A,\times, \texttt{1})$. Indeed, the counit $\varepsilon: c \to
\texttt{1}$ is the unique map $!_c$ to the terminal object
$\texttt{1}$, and the comultiplication $\delta: c \to c \times c$
is the diagonal morphism $\Delta_c$. This yields an isomorphism of
categories $\Comon(\A) \simeq \A$. Given an arbitrary object $c
\in \A$, we write $\overline{\textbf{c}}$ for the corresponding comonoid in
$(\A, \times, \texttt{1})$.

\begin{proposition} The assignment
$$(a, \theta_a : a \to a \times c)\longrightarrow (p_2 \cdot \theta_a : a \to
c)$$ yields an isomorphism of categories
$$^{\overline{\textbf{c}}}\A\simeq {\A\!\downarrow \!c},$$ where
$^{\overline{\textbf{c}}}\A=\A_{c \times -}$, while
$\A\!\downarrow \!c$ is the comma-category of objects over $c$,
that is, {\em objects} are morphisms $f: a \to c$ with codomain
$c$ and {\em morphisms} are
commutative diagrams $$
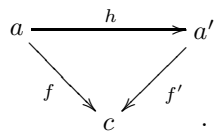
\xymatrix{a \ar[rd]_{f}\ar[rr]^h&& a'\ar[ld]^{f'}\\
&c&.}$$
\end{proposition}

If the category $\A$ has pullbacks, then for any morphism $f: c
\to d$ in $\A$, the functor $f_* : \A\!\downarrow \!c \to
\A\!\downarrow \! d$ given by the composition with $f$ has the
right adjoint $f^* : \A\!\downarrow \!d \to \A\!\downarrow \!c$
given by pulling back along the morphism $f$. Now, identifying $f:
c \to d$ with the morphism $f:\overline{\textbf{c}}\to
\overline{\textbf{d}}$ of the corresponding comonoids in $\A$, one
can see the functors $f^*$ and $f_*$ as the induction functor
$^{\overline{\textbf{c}}}\A\to ^{\overline{\textbf{d}}}\!\!\!\A$
and the coinduction functor $^{\overline{\textbf{d}}}\A \to
^{\overline{\textbf{c}}}\!\!\!\A$, respectively.  Given an object
$c \in \A$, we write $P_c$ and $U_c$ for the functors $(!_c)^*$
and $(!_c)_*$.

  Given a symmetric monoidal category $\bV=(V, \otimes, I)$,
the category $\Mon(\bV)$ of monoids in $\bV$ is again a
monoidal category. For two $\bV$-monoids $\textbf{A}=(A, m_A,
e_A)$ and $\textbf{B}=(B, m_B, e_B)$, their tensor product is
defined as
$$\textbf{A}\otimes\textbf{B}=(A \otimes B, (m_A
\otimes m_B)(1\otimes \tau_{A,B} \otimes 1), e_A \otimes e_B),$$
where $\tau$ is the symmetry in $\bV$. The unit object for this
tensor product is the trivial $\bV$-monoid $\textbf{I}=(I, 1_I,
1_I)$. Similarly, the category $\Comon(\bV)$ of $\bV$-comonoids
inherits, in a canonical way, the monoidal structure from $\bV$
making it a  monoidal category.

It is well-known that one can describe bimonoids in any symmetric
monoi\-dal category $\bV$ as monoids in the monoidal category of
comonoids in $\bV$. Thus, writing $\Bimon(\bV)$ for the category
of bimonoids in $\bV$, then $\Bimon(\bV)=\Mon(\Comon(\bV)).$ In
particular, since $\Comon(\A)\simeq\A$ for any cartesian monoidal
category $\A$, one has  $\Bimon(\A) = \Mon(\Comon(\A))
\simeq\Mon(\A)$. Thus, for any monoid $\textbf{b}=(b,m_b, e_b)$ in
$(\A, \times, \texttt{1})$, the 6-tuple
$$\widehat{\textbf{b}}=((b, m_b, e_b),(b,
\Delta_b, !_b))$$ is a bimonoid in $(\A, \times, \texttt{1})$. In
particular, then the functor $b \times - : \A \to \A$ is a
$(\tau_{b,b} \times -)$-bimonad (in the sense of \cite{MW}).

Fix now a monoid $\textbf{b}=(b,m_b, e_b)$ in $(\A, \times,
\texttt{1})$. Since $\widehat{\textbf{b}}$ is  a bimonoid in $(\A,
\times, \texttt{1})$, the category $ {_\mathbf{b}\A}:=\A_{b\times-}$ of
$\textbf{b}$-modules is monoidal. More precisely, if $(x,
\alpha_x), (y, \alpha_y)\in {_{\textbf{b}}\A}$, then their tensor
product is the pair $(x \times y, \alpha_{x \times y})$, where
$\alpha_{x \times y}$ is the composite
$$\xymatrix{b \times x \times y \ar[rr]^{\Delta_b \times x \times y}&&b \times b\times x \times y
\ar[rr]^{b \times \tau_{b, x} \times y}&& b \times x \times b \times
y \ar[rr]^{\alpha_x \times \alpha_y}&& x \times y.}$$
It is easy to see that this monoidal structure is cartesian and coincides with
the cartesian structure on ${_{\textbf{b}}\A}$ which can be lifted
from $\A$ along the forgetful functor ${_{\textbf{b}}\A} \to \A.$

Suppose now that $(c, \alpha_c : b \times c \to c) \in
{_{\textbf{b}}\A}$. Applying the previous proposition to the comonoid
$\overline{(c, \alpha_c)}$ in the cartesian monoidal category
${_{\textbf{b}}\A}$ gives

\begin{proposition} If $(c, \alpha_c ) \in {_\mathbf{b}\A}$, then the
assignment $$((x,\alpha_x), \theta_{(x,\alpha_x)})\longrightarrow
((x, \alpha_x),p_2\cdot \theta_{(x,\alpha_x)} )$$
yields an isomorphism of categories
$$^{\overline{(c, \alpha_c)}}({_{\mathbf{b}}\A})\simeq {_{\mathbf{b}}\A}\!\downarrow \!(c, \alpha_c).$$
\end{proposition}

We have seen that the data
$$\widetilde{\textbf{b}}=(\underline{\textbf{b}}=(b \times -, m_b \times -,
e_b \times -), \overline{\textbf{b}}=(b \times -, \Delta_b \times
-, !_b \times - )\,,\tau_{b,b} \times -)$$
define  a
$(\tau_{b,b}\times -)$-bimonad on $\A$ and considering $b$ as an
object of ${_{\textbf{b}}\A}$ via the multiplication $m_b : b
\times b \to b$, one obtains easily that the categories
$\A^b_b:=\A^{\overline{\textbf{b}}}_{\underline{\textbf{b}}}(\tau_{b,b}\times
-)$ (compare \ref{bimonad})
and $^{\overline{(b, m_b)}}({_\textbf{b}\A})$ are isomorphic.
Thus, by the previous proposition, the categories $\A^b_b$ and $
{_\mathbf{b}\A}\!\downarrow \!(b, m_b)$ are also isomorphic.

\begin{theorem} Assume that
\begin{rlist}
\item $\A$ has small limits,  or
\item $\A$ has colimits and the
functor $b \times -$ preserves them, or
\item $\A$ admits equalisers and $b \times -$ has a right adjoint, or
\item $\A$ admits coequalisers and $b \times -$ has a left adjoint.
\end{rlist}
Then the functor
$$K: \A \to {_\mathbf{b}\A}\downarrow (b, m_b), \quad
 a \longrightarrow (b \times a, p_1 : b \times a \to b),$$
is an equivalence of categories if and only if $\mathbf{b}$ is a
group.
\end{theorem}
\begin{proof}
It is easy to see that modulo the isomorphism $\A^b_b\simeq
_\textbf{b}\!\!\A\!\downarrow \!(b, m_b)$, the functor $K: \A \to
{_\mathbf{b}\A}\!\downarrow \!(b, m_b)$ can be identified with the
comparison functor $K: \A\to \A^b_b $, which by \ref{T.Bim} is an
equivalence of categories if and only if the bimonad
$\widetilde{\textbf{b}}$ has an antipode, which is the case if and
only if the $\A$-bimonoid $\widehat{\textbf{b}}$ has one, i.e.,
$\widehat{\textbf{b}}$ is a Hopf monoid in $(\A, \times,
\texttt{1})$. Now the result follows from the fact that in any
cartesian monoidal category, a Hopf algebra is nothing but a group
(see, for example, \cite[5.20]{W}).
\end{proof}

Consider now an object $(c, \alpha_c) \in {_\mathbf{b}\A}$. Since
$(c, \alpha_c)$ is a comonoid in the cartesian monoidal category $
({_\mathbf{b}\A}, \times, \texttt{1})$, the composite
$$\xymatrix{
b \times c \times - \ar[r]^-{\Delta_b\times c\times -}& b \times b
\times c \times - \ar[rr]^-{b \times \tau_{b,c}\times -}&& b
\times c \times b  \times- \ar[rr]^-{\alpha_c \times b \times -}&&
c \times b \times - }$$ is an entwining from the monad
$\textbf{T}_{\textbf{b}}=b \times -$ to the comonad
$\textbf{G}_{\overline{\textbf{c}}}=c \times -$. Then one has a
lifting $\widetilde{\textbf{T}_{\textbf{b}}}$ of the monad
$\textbf{T}_{\textbf{b}}$ along the forgetful functor
$^{\overline{\textbf{c}}}\!\A=\A \! \downarrow \! c \to \A.$ It is
easy to see that if $(x, f: x \to c) \in \A \!\downarrow \!c$,
then
$$\widetilde{T_{\textbf{b}}}(x, f)=(b \times x, \alpha_c \cdot (b
\times f): b \times x \to c).$$ We write $_{\textbf{b}}(\A
\!\downarrow \!c)$ for the category $(\A \!\downarrow
\!c)_{\widetilde{\textbf{T}_{\textbf{b}}}}$. It is also easy to
see that the functor
$$K : \A \to {_{\mathbf{b}}(\A \!\downarrow \!c)}
$$ that takes an object $a \in \A$ to the object $$(c \times a,
\alpha_c \times a : b \times c \times a \to c \times a)$$ makes the
diagram
$$\xymatrix{& {_{\mathbf{b}}(\A \!\downarrow \!c)}\ar[d]^{U}\\
\A \ar[r]_{P_c} \ar[ru]^{K}& \A \!\downarrow \!c}$$
commute, where
$U$ is the evident forgetful functor. Then the corresponding
$\widetilde{\textbf{T}_{\textbf{b}}}$-module structure on $P_c$ is
given by the morphism $ \alpha_c \times - :b \times c \times - \to
c \times -$. Since the forgetful functor $U_c : \A \!\downarrow
\!c \to \A$ that takes $f: x \to c$ to $x$ is left adjoint to the
functor $P_c$ and since the $(f: x \to c)$-component of the unit
of the adjunction $U_c \dashv P_c$ is the morphism $<f, 1_x>:x \to
c \times x$, the $(f: x \to c)$-component $t_f$ of the monad
morphism $t: \widetilde{\textbf{T}_{\textbf{b}}} \to P_cU_c $ is
the composite
$$\xymatrix{b \times x \ar[rr]^-{b \times <f, 1_x>}&& b
\times c \times x \ar[r]^-{\alpha_c \times 1_x}& c \times x.}$$ We
write $\gamma_c$ for the morphism $t_{1_c}: b \times c \to c \times c.$

One says that a morphism $f: a \to b$ in $\A$ is an
\emph{(effective) descent morphism} if the corresponding functor
$f^* : \A\!\downarrow \! b \to \A \! \downarrow \!a$ is
precomonadic (resp. monadic).

\begin{theorem} \label{Th.6.3}Let $\mathbf{b}=(b, m_b, e_b)$ be a monoid in $\A$ and
let $(c, \alpha_c)\in {_\mathbf{b}}\A$. Suppose that
  \begin{itemize}
    \item [(i)] $\A$ admits all small limits, or
    \item [(ii)] $\A$ admits coequalisers of reflexive pairs and
    the functors $b \times - :\A \to \A$ and $c \times - :\A \to \A$
    both have left adjoints.
  \end{itemize}
Then the functor $K : \A \to {_\mathbf{b}(\A \!\downarrow \!c)}$
is an equivalence of categories if and only if   $\gamma_c:b
\times c \to c \times c$ is an isomorphism and $!_c : c \to
\texttt{1}$ is an effective descent morphism.
\end{theorem}
\begin{proof}According to Proposition \ref{P.1.2}, the functor $K$ is an
equivalence of categories if and only if  the functor $P_c$ is
comonadic (i. e. if the morphism $!_c : c \to \texttt{1}$ is an
effective descent morphism) and $t:
\widetilde{\textbf{T}_{\textbf{b}}} \to P_cU_c $ is an isomorphism
of monads. Since the functors $b \times - :\A \to \A$ and $c
\times - :\A \to \A$ both preserve those limits that exist in
$\A$, it follows from \ref{P.2.5} that if $\A$ satisfies (i) or
(ii), $t$ is an isomorphism if and only if its restriction on free
$P_cU_c$-algebras is so. But any free $P_cU_c$-algebra has the
form $(c \times x, p_1)$ for some $x \in \A$ and it is not hard to
see that the $(c \times x, p_1)$-component $t_{(c \times x, p_1)}$
of $t$ is the morphism $\gamma_c \times x$. It follows that $t_{(c
\times x, p_1)}$ is an isomorphism for all $x \in \A$ if and only
if  the morphism $\gamma_c$ is an isomorphism. This completes the
proof.
\end{proof}

We call an object $a \in \A$ \emph{faithful} if the functor $a
\times - : \A \to \A$ is faithful. Note that $a$ is faithful if and only if
the unique morphism $!_a : a \to \texttt{1}$ is a descent morphism.

We follow Chase and Sweedler \cite{CS}  in calling an object $(c,
\alpha_c) \in {_{\textbf{b}}\A}$ a \emph{Galois
$\mathbf{b}$-object} if $c$ is a faithful object in $\A$ such that
the morphism $\gamma_c : b \times c \to c \times c$ is an
isomorphism. Using this notion, we can rephrase the previous
theorem as follows.

\begin{theorem} In the situation of the previous theorem, if $(c,
\alpha_c)\in {_{\mathbf{b}}\A}$ is a Galois $\mathbf{b}$-object,
then the functor $K: \A \to {_\mathbf{b}(\A \!\downarrow \!c)}$
is an equivalence of categories if and only if
 $!_c : c \to \texttt{1}$ is an effective descent morphism.
\end{theorem}

If any descent morphism in $\A$ is effective (as surely it is when
$\A$ is an exact category in the sense of Barr, see \cite{JT}),
then one has

\begin{corollary} If every descent morphism in $\A$ is
effective, then for any Galois $\mathbf{b}$-object $(c,
\alpha_c)$, the functor $K: \A \to {_\mathbf{b}(\A \!\downarrow\!c})$
is an equivalence of categories.
\end{corollary}

  Note that if $g: 1 \to G_{\bar{\textbf{c}}}$ is a grouplike morphism
for the comonad $G_{\bar{\textbf{c}}}$, then the composite
$\texttt{1}\xrightarrow {g_{\texttt{1}}}
G_{\bar{\textbf{c}}}(\texttt{1})=c \times \texttt{1} \xrightarrow
{p_2}\texttt{1}$ is the identity morphism, implying that the
morphism $!_c: c\to\texttt{1}$ is a split epimorphism. It is then
easy to see that the counit of the adjunction $U_c \dashv P_c$ is
a split epimorphism, and it follows from the dual of
\cite[Proposition 3.16]{Me} that the functor $P_c$ is monadic
(i.e., $!_c : c \to \texttt{1}$ is an effective descent morphism)
provided  that the category $\A$ is Cauchy complete. In the light
of the previous theorem, we get:

\begin{theorem} In the situation of Theorem \ref{Th.6.3}, if $\A$
is Cauchy complete and if there exists a grouplike morphism for
the comonad  $G_{\bar{\textbf{c}}}$, then the functor
$K: \A \to {_\mathbf{b}(\A \!\downarrow \!c)}$ is an equivalence of
categories if and only if $(c, \alpha_c)\in {_\mathbf{b}\A}$ is
a Galois $\mathbf{b}$-object.
\end{theorem}

Recall from  \cite{CS}  that an object $a \in \A$ is
\emph{(faithfully) coflat} if the functor $$a \times - : \A \to
\A$$ preserves coequalisers (resp. preserves and reflects
coequalisers).

\begin{theorem} Let $\A$ be a category with finite products and
coequalisers, and  $\mathbf{b}=(b, m_b, e_b)$ a  monoid in the
cartesian monoidal category $\A$ with $b$ coflat and let $(c,
\alpha_c)\in {_\mathbf{b}\A}$ be a $\mathbf{b}$-Galois object with
$!_c: C \to 1$ an effective descent morphism. Assume
\begin{itemize}
    \item [(i)] $\A$ admits all small limits, or
    \item [(ii)] the functors $b \times - :\A \to \A$ and $c \times - :\A \to \A$
    both have left adjoints.
  \end{itemize} Then $c$ is (faithfully) coflat.
\end{theorem}
\begin{proof} Note first that since $\A \!\downarrow \! c
\simeq ^{\overline{\textbf{c}}}\!\!\A$ and since the category $\A$
admits coequalisers, the category $\A\!\downarrow \! c$ also
admits coequalisers and the forgetful functor $U_c :
\A\!\downarrow \! c \to \A$ creates them. Now, if $b$ is coflat,
then the functor $b \times - : \A \to \A$ preserves coequalisers,
and it follows from the commutativity of the diagram
$$\xymatrix{\A\downarrow c \ar[d]_{U_c}\ar[r]^{\widetilde{T_{\textbf{b}}}}& \A\downarrow c \ar[d]^{U_c}\\
\A \ar[r]_{T_{\textbf{b}}=b \times -}& \A}$$ that the functor
$\widetilde{T_{\textbf{b}}}$ also preserves coequalisers. As in
the proof of \ref{Th.6.3}, one can show that the morphism $t:
\textbf{T}_{\textbf{b}} \to P_cU_c $ is an isomorphism of monads.
Thus, in particular, the monad $P_cU_c $ preserves coequalisers.
Since the morphism $!_c : c \to \texttt{1}$ is an effective
descent morphism by our assumption on $c$, the functor $P_c$ is
monadic. Applying now the dual of
 \cite[Proposition 3.11]{Me}, one gets that the functor $U_c P_c=c
\times -$ also preserves coequalisers. Thus $c$ is coflat.
\end{proof}

As a consequence, we have:
\begin{theorem}\label{Th.6.7} Let $\A$ be a category with finite products and
coequalisers in which all descent morphisms are effective. Suppose
that $\mathbf{b}=(b, m_b, e_b)$ is a  monoid in the cartesian
monoidal category $\A$ with $b$ coflat and that $(c, \alpha_c)\in
{_\mathbf{b}\A}$ is a $\textbf{b}$-Galois object. If
\begin{itemize}
    \item [(i)] $\A$ admits all small limits, or
    \item [(ii)] the functors $b \times - :\A \to \A$ and $c \times - :\A \to \A$
    both have left adjoints,
  \end{itemize} then $c$ is (faithfully) coflat.
\end{theorem}

\begin{thm}{\bf Opposite category of commutative algebras.} \em
Let $k$ be a commutative ring (with unit) and let $\A$ be the
opposite of the category of commutative unital $k$-algebras.

It is well-known that $\A$ has finite products and coequalisers.
If $\textbf{A}=(A, m_A, e_A)$ and $\textbf{B}=(B, m_B, e_B)$ are
objects of $\A$ (i.e. if $A$ and $B$ are commutative
$k$-algebras), then $A \otimes_k B$ with the obvious $k$-algebra
structure is the product of $\textbf{A}$ and $\textbf{B}$ in $\A$:
the projections $p_1: A \otimes_k B \to A$ and  $p_2: A \otimes_k
B \to B$ are given by $1_A \otimes_k e_B : A \to A \otimes_kB $
and $e_A \otimes_k 1_B:
 B \to  A \otimes_k B $, respectively. Furthermore, if $f, g: A \to B$
are morphisms in $\A$, then the pair $(C, i)$, where $C=\{b \in B
|f(b)=g(b)\}$ and $i:C \to B$ is the canonical embedding of
$k$-algebras, defines a coequaliser in $\A$. The terminal object in
$\A$ is $k$.

An object $A$ in $\A$ (i.e. a commutative $k$-algebra) is
(faithfully) coflat if and only if  $A$ is a (faithfully) flat
$k$-module (see, \cite{CS}). Moreover, a monoid in the cartesian
monoidal category $\A$ is a commutative $k$-bialgebra, which is a
group in $\A$ iff it has an antipode, and if $\textbf{B}$ is a
commutative $k$-bialgebra, then
$(C, \alpha_C) \in {_\textbf{B}\A}$ if and only if  $C$ is a commutative
$\textbf{B}$-comodule algebra.

Note that in the present context,
$(C, \alpha_C) \in {_\textbf{B}\A}$ is a Galois $\textbf{B}$-object if $C$ is a
faithful $k$-module and the composite
$$\xymatrix{\gamma_C : C \otimes_k C \ar[rr]^-{\alpha_C \otimes_k
C}&& B \otimes_k C \otimes_k C \ar[rr]^-{B \otimes_k m_C} & & B
\otimes_k C\,,}
$$ where $m_C :C\otimes_k C \to C$ is the multiplication in $C$,
is an isomorphism.
\end{thm}

  Since the category $\A$ admits all small limits and since in
$\A$ every descent morphism is effective (see \cite{Mes}), one can
apply Theorem \ref{Th.6.7} to deduce the following

\begin{theorem} Let $\textbf{B}$ be a commutative $k$-bialgebra with
$B$ a flat $k$-module. Then any Galois $\textbf{B}$-object in $\A$
is a faithfully flat $k$-module.
\end{theorem}

Note finally that when $\textbf{B}$ is a Hopf algebra which is
finitely generated and projective as a $k$-module, the result was
obtained by Chase and Sweedler, see  \cite[Theorem 12.5]{CS}.

\bigskip

{\bf Acknowledgements.} The work on this paper was started
during a visit of the first author at the Department of
Mathematics at the Heinrich Heine University of D\"usseldorf
supported by the German Research Foundation (DFG)
and continued with support by Volkswagen
Foundation  (Ref.: I/84 328 and GNSF/ST06/3-004).
The authors express their thanks to all these 
institutions.

\bigskip

\noindent
{\bf Addresses:} \\[+1mm]
{Razmadze Mathematical Institute, 1, M. Aleksidze st., Tbilisi
0193,  } {\small and} \\
 {Tbilisi Centre for Mathematical Sciences,
Chavchavadze Ave. 75, 3/35, Tbilisi 0168}, \\
 Republic of Georgia,
    {\small bachi@rmi.acnet.ge}\\[+1mm]
{Department of Mathematics of HHU, 40225 D\"usseldorf, Germany},
  {\small wisbauer@math.uni-duesseldorf.de}

\end{document}